\documentclass[letterchapter,11pt]{article}

\usepackage{latexsym}
\usepackage{amssymb}
\usepackage{amsmath}
\usepackage{graphicx}
\usepackage[english]{babel}

\usepackage{fontenc}
\usepackage{amsfonts}
\usepackage{amsmath}
\usepackage{amsthm}
\usepackage{enumerate}
\usepackage{textcomp}
\usepackage{geometry}
\usepackage{mathrsfs}
\usepackage{natbib}
\usepackage{fancyhdr}

\usepackage[latin1]{inputenc}

\usepackage{amsfonts}
\usepackage{amsthm}
\usepackage{latexsym}

\newcommand{\RR}{{\rm I}\kern-0.18em{\rm R}}
\newcommand{\h}{{\rm I}\kern-0.18em{\rm H}}
\newcommand{\K}{{\rm I}\kern-0.18em{\rm K}}

\newcommand{\1}{{\rm 1}\kern-0.22em{\rm I}}

\font\tenmath=msbm10 \font\sevenmath=msbm7 \font\fivemath=msbm5
\newfam\mathfam \textfont\mathfam=\tenmath
\scriptfont\mathfam=\sevenmath \scriptscriptfont\mathfam=\fivemath

\font\teneusb=eusb10 \font\seveneusb=eusb7 \font\fiveeusb=eusb5
\newfam\eusbfam \textfont\eusbfam=\teneusb
\scriptfont\eusbfam=\seveneusb \scriptscriptfont\eusbfam=\fiveeusb
\def\eusb{\fam\eusbfam}
\def\frak R{\eusb R}

\font\tenams=msam10 \font\sevenams=msam7 \font\fiveams=msam5
\newfam\amsfam \textfont\amsfam=\tenams
\scriptfont\amsfam=\sevenams \scriptscriptfont\amsfam=\fiveams

\newtheorem{rem}{Remark}

\newtheorem{corr}{Corollary}
\newtheorem{lem}{Lemma}

\newtheorem{theo}{Theorem}
\newtheorem{defi}{Definition}
\newtheorem{fact}{Fact}

    \font\tenbifull=cmmib10 
    \font\tenbimed=cmmib7
    \font\tenbismall=cmmib5
\mathchardef\bbGamma="7000 \mathchardef\bbDelta="7001
\mathchardef\bbPhi="7002 \mathchardef\bbAlpha="7003
\mathchardef\bbXi="7004 \mathchardef\bbPi="7005
\mathchardef\bbSigma="7006 \mathchardef\bbUpsilon="7007
\mathchardef\bbTheta="7008 \mathchardef\bbPsi="7009
\mathchardef\bbOmega="700A \mathchardef\bbalpha="710B
\mathchardef\bbbeta="710C \mathchardef\bbgamma="710D
\mathchardef\bbdelta="710E \mathchardef\bbepsilon="710F
\mathchardef\bbzeta="7110 \mathchardef\bbeta="7111
\mathchardef\bbtheta="7112 \mathchardef\bbiota="7113
\mathchardef\bbkappa="7114 \mathchardef\bblambda="7115
\mathchardef\bbmu="7116 \mathchardef\bbnu="7117
\mathchardef\bbxi="7118 \mathchardef\bbpi="7119
\mathchardef\bbrho="711A \mathchardef\bbsigma="711B
\mathchardef\bbtau="711C \mathchardef\bbupsilon="711D
\mathchardef\bbphi="711E \mathchardef\bbchi="711F
\mathchardef\bbpsi="7120 \mathchardef\bbomega="7121
\mathchardef\bbvarepsilon="7122 \mathchardef\bbvartheta="7123
\mathchardef\bbvarpi="7124 \mathchardef\bbvarrho="7125
\mathchardef\bbvarsigma="7126 \mathchardef\bbvarphi="7127

\usepackage[colorlinks=true,linkcolor=red,citecolor=blue]{hyperref}
\date{}

\numberwithin{equation}{section}

\begin{document}

\author{Dimbihery Rabenoro\\ \scriptsize{INRIA Sophia Antipolis, France}}

\title{A conditional limit theorem for independent random variables}

\maketitle

\section{Introduction}

\subsection{Context and Scope}

Let $(X_{j})_{j \geq 1}$ be a sequence of independent random variables  valued in $\mathbb{R}$, of common support $\mathcal{S}_{X}$. Assume that for all $j \geq 1$, $X_{j}$ is absolutely continuous with respect to the Lebesgue measure. Let $a \in \mathcal{S}_{X}$. For $n \geq 1$ and $1 \leq k < n$, let $Q_{nak}$ be the conditional distribution of $X_{1}^{k} := (X_{j})_{1 \leq j \leq k}$ given $\left\{ \frac{S_{1,n}}{n} = a \right\}$, where $S_{1,n} := \sum\limits_{j=1}^{n} X_{j}$. Under suitable conditions on the distributions of the $X_{j}$'s, we have obtained in \cite{Rabenoro 2018} that for any sequence $(k_{n})$ such that, as $n \rightarrow \infty$, $k_{n} \rightarrow \infty$ and $k_{n}= \circ(n)$,   
\begin{equation*}
d_{TV} \left( Q_{nak_{n}}~;~\widetilde{P}_{1}^{k_{n}} \right) 
= \mathcal{O} \left( \frac{k_{n}}{n} \right), 
\end{equation*}

\noindent
where $d_{TV}$ is the total variation distance and $\widetilde{P}_{1}^{k_{n}}$ is a product of Gibbs measures $\gamma_{j}^{a}$, $1 \leq j \leq k_{n}$. Then, a natural question arises: What can be said about the distribution of the $n-k_{n}$ other r.v.'s, that is of $(X_{j})_{ k_{n}+1 \leq j  \leq n}$, given $\left\{ \frac{S_{1,n}}{n} = a \right\}$ ? Set $k_{n}':=n-k_{n}$, so that $\frac{k_{n}'}{n} \rightarrow 1$ as $n \rightarrow \infty$.

\noindent\\
In this paper, we thus study the asymptotic distribution of $Q_{nak_{n}}$ if $\frac{k_{n}}{n}$ is allowed to converge to 1 as $n \rightarrow \infty$. As expected (see \cite{Dembo and Zeitouni 1996}), we do not obtain here a Gibbs type measure as an approximation of $Q_{nak_{n}}$, since the condition $\frac{k_{n}}{n} = \circ(n)$ is neccessary for that.

\subsection{Importance Sampling framework}

Consider a sequence $(X_{j})_{j \geq 1}$ of r.v.'s. For large $n$ but \textit{fixed}, one intends to estimate 
\begin{equation}
 \Pi_{n} := P(X_{1}^{n} \in \mathcal{E}_{n}), \quad \textrm{for some event} \enskip \mathcal{E}_{n}. 
\end{equation}

\noindent
Let $p_{1}^{n}$ be the density of $X_{1}^{n}$. For any density $q_{1}^{n}$ on $\mathbb{R}^{n}$, the classical Importance Sampling (IS) estimator of $\Pi_{n}$ is the following: 
\begin{equation} 
\label{ISestimator}
\widehat{\Pi}_{n}(N) := \frac{1}{N} \sum\limits_{i=1}^{N} \frac{p_{1}^{n}(Y_{1}^{n}(i))}{q_{1}^{n}(Y_{1}^{n}(i))} 
1_{\mathcal{E}_{n}} (Y_{1}^{n}(i)), 
\end{equation}

\noindent
where $(Y_{1}^{n}(i))$ are i.i.d copies of a random vector $Y_{1}^{n}$ sampled from $q_{1}^{n}$. Then, the law of large numbers insures that $\widehat{\Pi}_{n}(N)$ converges almost surely to $\Pi_{n}$, as $N \rightarrow \infty$. Compared to the usual Monte Carlo method, the interest of this resampling procedure is to reduce the variance of the resulting estimator. It is well known that the density $q_{1}^{n}$ for which the variance of $\widehat{\Pi}_{n}(N)$ is minimal is the conditional density $p(X_{1}^{n} | \mathcal{E}_{n})$. However, this density involves $\Pi_{n}$ itself. So, it is natural to search an  approximation of $p(X_{1}^{n} | \mathcal{E}_{n})$, which could be approached itself by $p(X_{1}^{k_{n}} | \mathcal{E}_{n})$, for $k_{n}$ satisfying $\frac{k_{n}}{n} \rightarrow 1$ as $n \rightarrow \infty$.

\noindent\\
This method has been developed in \cite{Broniatowski and Ritov 2009}, for an i.i.d. sequence  $(X_{j})_{j \geq 1}$ of centered r.v.'s, with 
\begin{equation}
{\mathcal{E}_{n}} = \left\{  (x_{i})_{1 \leq i \leq n} \in \mathbb{R}^{n} : \sum\limits_{i=1}^{n} x_{i} \geq na_{n}  \right\},
\end{equation}

\noindent
for some sequence $(a_{n})$ converging slowly to 0. Therefore, $\widehat{\Pi}_{n}(N)$ estimates the moderate deviation probability of $\frac{S_{1,n}}{n}$. Let $Q'_{n}$ be the distribution of $X_{1}^{k_{n}}$ given $S_{1,n} \geq na_{n}$. Then, it is established in \cite{Broniatowski and Ritov 2009} that, for some density $g_{k_{n}}$ on $\mathbb{R}^{k_{n}}$, 
\begin{equation}
\label{theoRitov}
p\left( \left. X_{1}^{k_{n}} = Y_{1}^{k_{n}} \right|  S_{1,n} \geq na_{n} \right) \approx g_{k_{n}}(Y_{1}^{k_{n}}), 
\quad \textrm{where} \quad
Y_{1}^{k_{n}} \textrm{ is sampled from } Q'_{n}. 
\end{equation}

\noindent
The precise sense of $\approx$ is given in Section $\ref{landauNotations}$ below. One deduces from a reverse lemma that 
\begin{equation}
g_{k_{n}}(Z_{1}^{k_{n}}) \approx p\left( \left. X_{1}^{k_{n}} = Z_{1}^{k_{n}} \right|  S_{1,n} \geq na_{n} \right), 
\quad \textrm{where } 
Z_{1}^{k_{n}} \textrm{ is sampled from } g_{k_{n}}. 
\end{equation}

\noindent
Now, it turns out that $g_{k_{n}}$ has a computable expression, which allows to simulate $Z_{1}^{k}$. From $g_{k_{n}}$, one constructs a density 
$\overline{g}_{n}$ on $\mathbb{R}^{n}$ which is also computable. In $(\ref{ISestimator})$, $q_{1}^{n}$ and $(Y_{1}^{n}(i))$ are replaced respectively by $\overline{g}_{n}$ and copies of a r.v. with density $\overline{g}_{n}$. The IS estimator obtained has better performances than the existing ones which estimate $\Pi_{n}$. Now, it is reasonable to expect that $(\ref{theoRitov})$ implies that the distribution of $X_{1}^{k_{n}}$ given $\left\{S_{1,n} \geq na_{n} \right\}$ is close to the distribution associated to $g_{k_{n}}$.

\subsection{Our result}

\textit{In the sequel, we simply write $k$ instead of $k_{n}$}.

\noindent\\
We can use these IS ideas to get an approximation of $Q_{nak}$ for some $k$ such that $\frac{k}{n} \rightarrow 1$ (see Theorem $\ref{largekTheorem}$), but also for a class of $k$ which are some $o(n)$ (see Theorem $\ref{smallkTheorem}$). However, in both cases, the condition $n-k \rightarrow \infty$ is required for the Edgeworth expansions.  

\noindent\\
We consider a sequence $(X_{j})_{j \geq 1}$ of \textit{independent} r.v.'s. For any $a \in \mathcal{S}_{X}$, let $p\left( \left. X_{1}^{k} = \cdot \right|  S_{1,n} = na \right)$ be the density of $X_{1}^{k}$ given $\left\{ S_{1,n} = na \right\}$. In this paper, we obtain that, for some density $g_{k}$ on $\mathbb{R}^{k}$,
\begin{equation}
p\left( \left. X_{1}^{k} = Y_{1}^{k} \right|  S_{1,n} = na \right) \approx g_{k}(Y_{1}^{k}), 
\quad \textrm{where} \quad
Y_{1}^{k} \sim \mathcal{L} \left( \left. X_{1}^{k}  \right|  S_{1,n} = na \right). 
\end{equation}

\noindent
We deduce (see Section 2.4) that 
\begin{equation}
\left\| Q_{nak} -  G_{k}  \right\|_{TV} \longrightarrow 0 
\quad \textrm{as} \enskip n \rightarrow \infty, 
\end{equation}

\noindent
where $G_{k}$ is the distribution associated to $g_{k}$. More precisely, when $k$ is small ($k = o(n^{\rho})$ with 
$0 < \rho < 1/2$), $G_{k}$ is the same Gibbs type measure as in the preceding chapter, while for large $k$ (see the assumptions of Theorem $\ref{largekTheorem}$), $G_{k}$ is a slight modification of this measure. 

\noindent\\
Kolmogorov's extension theorem does not apply to the sequence $(Q_{nan})_{n \geq 1}$ of probability measures. Therefore, we need to consider a sequence $((\Omega_{n},\mathcal{A}_{n}, \mathcal{P}_{n}))_{n \geq 1}$ of probability spaces s.t. for any $n \geq 1$, $Y_{1}^{n}$ is a random vector defined on $(\Omega_{n},\mathcal{A}_{n}, \mathcal{P}_{n})$ and the distribution of $Y_{1}^{n}$ is $Q_{nan}$. Then, for $k \leq n$, $Q_{nak}$ is the distribution of $Y_{1}^{k}$. The properties of $(Y_{1}^{n})_{n \geq 1}$ are studied in Section 3, after some elementary results and statement of the Assumptions in Section 2, while Section 4 is devoted to our main Results and their proofs.

\section{Assumptions and elementary results}

All the r.v.'s considered are a.c. w.r.t. the Lebesgue measure on $\mathbb{R}$. For any r.v. $X$, let $P_{X}$ be its distribution, $p_{X}$ its density and $\Phi_{X}$ its moment generating function (mgf). For any $j \geq 1$, set 
\begin{equation}
P_{j} := P_{X_{j}} \quad ; \quad p_{j} := p_{X_{j}} \quad ; \quad 
\Phi_{j} := \Phi_{X_{j}}. 
\end{equation}

\subsection{Conditional density}

Let $U$ and $V$ be r.v.'s having respective densities $p_{U}$ and $p_{V}$ and a joint density denoted by $p_{(U,V)}$. Then, there exists a conditional density of $U$ given $V$, denoted as follows. 
\begin{equation*}
p\left( \left. U = u \right| V = v \right) = 
\frac{ p_{(U,V)} \left(u, v \right)}{p_{V}(v)}. 
\end{equation*}

\begin{fact}
\label{densCond}
Let $(X_{j})_{j \geq 1}$ be a sequence of independent r.v.'s. For any $n \geq 1$ and $1 \leq i \leq n$, let $J_{n}$ be a subset of $\left\{ i, ..., n \right\}$ s.t. $\alpha_{n} := \left| J_{n} \right| < n-i+1$. Let $L_{n}$ be the complement of $J_{n}$ in $\left\{ i, ..., n \right\}$. Set $S_{L_{n}} := \sum\limits_{j \in L_{n}} X_{j}$. Then, there exists a conditional density of $(X_{j})_{j \in J_{n}}$ given $S_{i,n}$, defined by 
\begin{equation}
\label{densCondFormule}
p\left( ( \left. X_{j} )_{j \in J_{n}} = (x_{j}) \right| S_{i,n} = s \right) = 
\frac{ \left\{ \prod\limits_{j \in J_{n}} p_{j}(x_{j}) \right\} 
p_{S_{L_{n}}} \left( s - \sum\limits_{j \in J_{n}} x_{j} \right) }
{p_{S_{i,n}}\left( s \right)}, 
\end{equation}
\end{fact}

\subsection{The tilted density}

\begin{defi}
For a r.v. $X$, let $\Phi_{X}$ be its mgf and let  
$\Theta_{X} := \left\{ \theta \in \mathbb{R} : \Phi_{X}(\theta ) < \infty \right\}$. For any $\theta \in \Theta_{X}$, denote by $\widetilde{X}^{\theta }$ a random vector having the tilted density, defined by 
\begin{equation}
p_{\widetilde{X}^{\theta}}(x) := \frac{(\exp \theta x) p_{X}(x)}{\Phi_{X}(\theta)}. 
\end{equation}
\end{defi}

\noindent\\
For any $j \geq 1$, set $\Phi_{j} := \Phi_{X_{j}}$. We suppose throughout the text that the functions $(\Phi_{j})_{j \geq 1}$ have the same domain of finiteness denoted by $\Theta$, which is assumed to be of non void interior. We write, for any $j \geq 1$, 
\begin{equation*}
\Theta := \left\{ \theta \in \mathbb{R}^{d} : \Phi_{j} ( \theta) < \infty \right\}. 
\end{equation*}

\bigskip

\begin{fact}
For any $j \geq 1$, there exists a probability space 
$(\Omega^{\theta}, \mathcal{A}^{\theta}, \mathcal{P}^{\theta})$ such that for all finite subset $J \subset \mathbb{N}$ and for all 
$(B_{j})_{j \in J} \in \mathcal{B}(\mathbb{R})^{|J|}$, 
\begin{equation}  
\mathcal{P}^{\theta} \left( \left(\widetilde{X}_{j}^{\theta}\right)_{j \in J} \in 
(B_{j})_{j \in J} \right) =
\prod_{j \in J} \widetilde{P}_{j}^{\theta}(B_{j}) =
\prod_{j \in J} \int\limits_{B_{j}} \widetilde{p}_{j}^{\theta}(x)dx, 
\end{equation}

\noindent
where $\widetilde{P}_{j}^{\theta} := P_{\widetilde{X}_{j}^{\theta}}$ and 
$\widetilde{p}_{j}^{\theta} := p_{\widetilde{X}_{j}^{\theta}}$. In other words, 
$\left(\widetilde{X}_{j}^{\theta}\right)_{j \geq 1}$ is a sequence of independent r.v.'s defined on $(\Omega^{\theta}, \mathcal{A}^{\theta}, \mathcal{P}^{\theta})$. 
\end{fact}

\bigskip

\begin{fact}
For any $j \geq 1$, and $\theta \in \Theta$, we have that  
\begin{equation}
\mathbb{E} \left[ \widetilde{X}_{j}^{\theta} \right] = m_{j}(\theta) 
\quad \textrm{where} \quad
m_{j}(\theta) := \frac{d\kappa_{j}}{d\theta}(\theta)
\enskip \textrm{and} \enskip
\kappa_{j}(\theta) := \log \Phi_{j}(\theta).
\end{equation}
\end{fact}

\bigskip

\begin{fact}
For any $\theta \in \Theta$, $j \geq 1$ and $j' \geq 1$,
\begin{equation}
\mathbb{E}\left[ \widetilde{X_{j}+X_{j'}}^{\theta} \right] = 
\mathbb{E}\left[ \widetilde{X}_{j}^{\theta}+\widetilde{X}_{j'}^{\theta} \right].
\end{equation}

\end{fact}

\bigskip

\begin{corr}
For any $n \geq 1$ and $1 \leq \ell \leq n$, for any $\theta \in \Theta$, 
\begin{equation}
\mathbb{E}\left[ \widetilde{S_{\ell,n}}^{\theta} \right] = 
\sum\limits_{j=\ell}^{n} m_{j}(\theta). 
\end{equation}

\end{corr}

\bigskip

\begin{fact}
\noindent\\
For any $j \geq 1$ and $\theta \in \Theta$, set 
\begin{equation*}
\overline{X}_{j}^{\theta} := \widetilde{X}_{j}^{\theta} - \mathbb{E}[\widetilde{X}_{j}^{\theta}] = \widetilde{X}_{j}^{\theta} - m_{j}(\theta)
\end{equation*}

\noindent
and for any $\ell \geq 3$, 
\begin{equation*}
s_{j}^{2}(\theta) := Var \left( \widetilde{X}_{j}^{\theta} \right) \quad ; \quad 
\sigma_{j}(\theta) := \sqrt{s_{j}^{2}(\theta)} \quad ; \quad
\mu_{j}^{\ell}(\theta) := \mathbb{E} \left[ \left(\overline{X}_{j}^{\theta} \right)^{\ell} \right] \quad ; \quad
|\mu|_{j}^{\ell}(\theta) := \mathbb{E} \left[ \left|\overline{X}_{j}^{\theta} \right|^{\ell} \right]. 
\end{equation*}

\noindent
Then, 
\begin{equation}
s_{j}^{2}(\theta) = \frac{d^{2} \kappa_{j}}{d\theta^{2}}(\theta) \quad \textrm{and} \quad
\mu_{j}^{\ell}(\theta) = \frac{d^{\ell} \kappa_{j}}{d \theta^{\ell}}(\theta). 
\end{equation}
\end{fact}

\bigskip

\subsection{Landau Notations} \label{landauNotations}

\begin{defi}
Let $(X_{n})_{n \geq 1}$ be a sequence of r.v.'s such that for any $n \geq 1$, $X_{n}$ is defined on a probability  space $(\Omega_{n},\mathcal{A}_{n}, \mathcal{P}_{n})$. Let $(u_n)$ be a  sequence of real numbers. We say that 

\noindent\\
$(X_{n})_{n \geq 1}$ is a $\mathcal{O}_{\mathcal{P}_{n}}(u_n)$ if for all $\epsilon > 0$,  there exists $A \geq 0$ and $ N_{\epsilon} \in \mathbb{N}$, s.t. for all $n \geqslant N_{\epsilon}$,
\begin{equation}
\mathcal{P}_{n} \left(\left| \frac{X_{n}} {u_{n}} \right|  \leqslant A \right) \geqslant 1 - \epsilon. 
\end{equation}

\noindent
$(X_{n})_{n \geq 1}$ is a $o_{\mathcal{P}_{n}}(u_n)$  if for all $\epsilon > 0$ and $\delta > 0$, there exists 
$N_{\epsilon, \delta} \in \mathbb{N}$ s.t. for all $n \geqslant N_{\epsilon}$,  
\begin{equation}
\mathcal{P}_{n} \left(\left| \frac{X_{n}} {u_{n}} \right|  \leqslant \delta \right) \geqslant 1 - \epsilon.
\end{equation}

\noindent
$(X_{n})_{n \geq 1}$ converges to $\ell \in \mathbb{R}$ in $\mathcal{P}_{n}$- probability and we note 
$X_{n} \enskip {\underset{\mathcal{P}_{n}} {\longrightarrow}} \enskip \ell $ if 
\begin{equation}
X_{n} = \ell + o_{\mathcal{P}_{n}}(1). 
\end{equation}
\end{defi}

\begin{rem}
These notations differ from the classical Landau notations in probability by the fact that here, the rv's $(X_{n})$ are not defined on the same probability space. However, they satisfy similar properties, which we will use implicitly in the proofs. 
\end{rem}

\subsection{A criterion for convergence in Total Variation Distance}

\begin{defi}
Set 
\begin{equation*}
\mathcal{A}_{\rightarrow 1} := 
\left\{ (B_{n})_{n \geq 1} \in \prod\limits_{n \geq 1} \mathcal{A}_{n} \enskip : \enskip
\mathcal{P}_{n}(B_{n}) \enskip {\underset{n \infty} {\longrightarrow}} \enskip 1 \right\}. 
\end{equation*}
\end{defi}

\begin{lem} 
\label{lienISdvt}
For all integer $n \geq 1$, let $Y_{1}^{n} : (\Omega_{n},\mathcal{A}_{n}, \mathcal{P}_{n}) \longrightarrow (\mathbb{R}^n, \mathcal{B}(\mathbb{R}^{n}))$ be a random vector. For any 
$1 \leq k \leq n$, the distribution of $Y_{1}^k$ is denoted by $P_{k}$. Let $G_{k}$ be a probability measure on $\mathbb{R}^{k}$. Assume that $P_{k}$ and $G_{k}$ have positive densities $p_{k}$ and $g_{k}$, and that $k \rightarrow \infty$ as $n \rightarrow \infty$. If there exists $(B_{n})_{n \geq 1} \in \mathcal{A}_{\rightarrow 1}$ s.t. for any $n \geq 1$, we have on $B_{n}$ that  
\begin{equation}
\label{errel}
p_{k}(Y_{1}^k) = g_{k}(Y_{1}^k) \left[1 + T_{n} \right] \enskip \textrm{where} \enskip 
T_{n} = o_{\mathcal{P}_{n}}(1), 
\end{equation}

\noindent
then, 
\begin{equation}
\left\| P_k - G_k \right\|_{TV} \enskip {\underset{n \infty} {\longrightarrow}} \enskip 0. 
\end{equation}

\end{lem}

\begin{proof}
For any $\delta >0$, set 
\begin{equation}
E(n,\delta):= \left\{ (y_{1}^{k}) \in \mathbb{R}^{k} : \left| \frac{p_k(y_{1}^{k})}{g_k(y_{1}^{k})} - 1 \right| \leqslant \delta \right\}. 
\end{equation}

\noindent 
Then, 
\begin{align*}
\mathcal{P}_{n} \left( \left\{ \left| T_{n} \right| \leqslant \delta \right\} \cap 
B_{n} \right) &\leqslant 
\mathcal{P}_{n} \left( \left| \frac{p_k(Y_{1}^{k})}{g_k(Y_{1}^{k})} - 1 \right| 
\leqslant \delta \right) \\
&= P_k(E(n,\delta)) \\
&= \int\limits_{E(n,\delta)} \frac{p_k(y_{1}^{k})}{g_k(y_{1}^{k})}g_k(y_{1}^{k})dy_{1}^{k} \\ 
&\leqslant 
(1+\delta)G_k(E(n,\delta)). 
\end{align*}

\noindent
By $(\ref{errel})$, for all $n$ large enough,   
\begin{align*}
\mathcal{P}_{n} \left( \left\{ \left| T_{n} \right| \leqslant \delta \right\} \cap B_{n} 
\right) &\geq 
1 - \mathcal{P}_{n}\left( \left\{ \left| T_{n} \right| > \delta \right\} \right) - 
\mathcal{P}_{n}(B_{n}^{c}) \\
&\geq 1-2\delta.
\end{align*}

\noindent
Combining the preceding inequalities, we obtain that for all $n$ large enough, 
\begin{equation}
1-2\delta \leq P_k(E(n,\delta)) \leq (1 + \delta) G_k(E(n,\delta)). 
\end{equation}

\noindent 
Therefore, 
\begin{equation}
\sup\limits_{C \in \mathcal{B}(\mathbb{R}^k)} | P_k(C) - P_k(C \cap E(n,\delta))| \leqslant P_k(E(n,\delta)^{c}) \leq 2\delta
\end{equation}

\noindent 
and 
\begin{align*}
\sup\limits_{C \in \mathcal{B}(\mathbb{R}^k)} | G_k(C) - G_k(C \cap E(n,\delta)) | &\leq 
1 - G_k(E(n,\delta)) \\
&\leq 1 - \frac{1-2\delta}{1+ \delta} \\
&= \frac{3\delta}{1+ \delta}. 
\end{align*}

\noindent
Now, we have that 
\begin{equation}
\sup\limits_{C \in \mathcal{B}(\mathbb{R}^k)} | P_k(C \cap E(n,\delta)) - G_k(C \cap E(n,\delta)) | \leqslant 
\sup\limits_{C \in \mathcal{B}(\mathbb{R}^k)} \int\limits_{C \cap E(n,\delta)} | p_k(y_{1}^{k}) - g_k(y_{1}^{k}) |  dy_{1}^{k}
\end{equation}

\noindent
From the definition of $E(n,\delta)$, we deduce that 
\begin{align*}
\sup\limits_{C \in \mathcal{B}(\mathbb{R}^k)} | P_k(C \cap E(n,\delta)) - G_k(C \cap E(n,\delta)) | &\leqslant  
\delta \sup\limits_{C \in \mathcal{B}(\mathbb{R}^k)} \int\limits_{C \cap E(n,\delta)} g_k(y_{1}^{k}) dy_{1}^{k} \\
&\leqslant \delta.
\end{align*}

\noindent
Finally, applying the triangle inequality, we have that for all $n$ large enough, 
\begin{align*}
\sup\limits_{C \in \mathcal{B}(\mathbb{R}^k)} | P_k(C) - G_k(C)| &\leq
2\delta + \delta + \frac{3\delta}{1+ \delta} \\
&= 3\delta \left( \frac{2+\delta}{1+\delta}\right), 
\end{align*}

\noindent
which converges to 0 as $\delta \rightarrow 0$. 
\end{proof}

\begin{rem}
A rate of convergence is not obtainable by this method.
\end{rem}

\subsection{A first calculus} \label{firstCalculus}

Set
\begin{equation}
p_{k}\left(Y_{1}^k\right) := p\left( \left. X_{1}^k = Y_{1}^k \right| S_{1,n} = na\right). 
\end{equation}

\noindent\\
First, we have that
\begin{equation}
p\left( \left.X_{1}^k = Y_{1}^k \right| S_{1,n} = na\right) =
p\left( \left. X_k=Y_k \right| X_1^{k-1}=Y_1^{k-1};S_{1,n}=na\right)
p\left( \left. X_1^{k-1}=Y_1^{k-1} \right| S_{1,n}=na\right)
\end{equation}

\noindent
Set $p_{k}\left(Y_1^{k}\right) := p\left( \left. X_{1}^k = Y_{1}^k \right| S_{1,n} = ns\right)$, then we deduce by induction on $k$ that 
\begin{equation}
\label{pkFirstStep}
p_{k}\left(Y_1^k\right) = \left\{ \prod_{i=1}^{k-1} 
p\left( \left. X_{i+1}=Y_{i+1} \right| X_1^i= Y_1^i ; S_{1,n}=na \right) \right\}
p\left( \left.X_{1}=Y_{1} \right| S_{1,n}=na\right).
\end{equation}

\noindent
For $1 \leq i_{1} \leq i_{2} \leq n$, set 
$\Sigma_{i_{1},i_{2}} := \sum\limits_{j=i_{1}}^{i_{2}} Y_{j}$. We deduce from 
$(\ref{pkFirstStep})$ that 
\begin{equation}
p_{k}\left(Y_1^k\right) = \left\{ \prod_{i=1}^{k-1} p\left(\left. X_{i+1}=Y_{i+1} \right| S_{i+1,n}=na-\Sigma_{1,i}\right) \right\} p\left(X_{1}=Y_{1} | S_{1,n}=na\right).
\end{equation}

\noindent
Let $\Sigma_{1,0}=0$. Then, 
\begin{equation}
p_{k}\left(Y_1^k\right) = \prod_{i=0}^{k-1} \pi_i, \enskip \textrm{where} \enskip
\pi_i := p\left(X_{i+1}=Y_{i+1} | S_{i+1,n}=na-\Sigma_{1,i}\right). 
\end{equation}

\noindent
The conditioning event being $\{ S_{i+1,n}=na-\Sigma_{1,i} \}$, we search $\theta$ s.t.
\begin{equation}
E\left[\widetilde{S_{i+1,n}}^{\theta}\right] = 
\sum\limits_{j=i+1}^{n} m_{j}(\theta) = 
na - \Sigma_{1,i}.
\end{equation}

\noindent
Since $\mathcal{P}_{n}$-a.s., $\Sigma_{1,i} + \Sigma_{i+1,n} = na$, this is equivalent to solve the following equation, where $\theta$ is unknown.  
\begin{equation}
\label{equaTilting}
\overline{m}_{i+1,n}(\theta) := \frac{\sum\limits_{j=i+1}^{n} m_{j}(\theta)}{n-i} = 
\frac{\Sigma_{i+1,n}}{n-i}.  
\end{equation}

\noindent
\textit{We will see below (see Definition $\ref{deftin}$) that, under suitable assumptions, equation $(\ref{equaTilting})$ has a unique solution $t_{i,n}$. In the following lines, the tilted densities pertain to $\theta = t_{i,n}$}.

\noindent\\
For $e=1,2$, let $\overline{q}_{i+e,n}$ be the density of $\overline{S}_{i+e,n}$, where 
\begin{equation}
\label{SienBarre}
\overline{S}_{i+e,n} := 
\frac{\widetilde{S_{i+e,n}} - \mathbb{E}\left[ \widetilde{S_{i+e,n}} \right]}
{\sqrt{Var\left(\widetilde{S_{i+e,n}}\right)}} = 
\frac{\widetilde{S_{i+e,n}} - \sum\limits_{j=i+e}^{n} m_{j}(t_{i,n})}
{\sqrt{\sum\limits_{j=i+e}^{n} s_{j}^{2}(t_{i,n})}}. 
\end{equation}

\noindent
Using the invariance of the conditional density under the tilting operation, Fact $\ref{densCond}$ and then renormalizing, we obtain that  
\begin{equation}
\pi_{i} = p\left(\widetilde{X_{i+1}}=Y_{i+1} | \widetilde{S_{i+1,n}}=na-\Sigma_{1,i} \right) = \widetilde{p}_{i+1}(Y_{i+1}) \frac{\sigma_{i+1,n}}{\sigma_{i+2,n}} 
\frac{\overline{q}_{i+2,n}(Z_{i+1})}{\overline{q}_{i+1,n}(0)}, 
\end{equation}

\noindent
where
\begin{equation*}
Z_{i+1} := \frac{m_{i+1} - Y_{i+1}}{\sigma_{i+2,n}}.
\end{equation*}

\subsection{Assumptions}

\begin{defi}
Let $f : (\alpha, \beta) \longrightarrow (A,B)$ be a function, where $\alpha$, $\beta$, $A$ and $B$ may be finite or not. Consider the following condition $(\mathcal{H})$. 

\noindent\\
$(\mathcal{H})$ : $f$ is strictly increasing \enskip and \enskip
$\lim\limits_{\theta \rightarrow \alpha} f(\theta) = A$ \enskip ; \enskip 
$\lim\limits_{\theta \rightarrow \beta} f(\theta) = B$.
\end{defi}

\subsubsection{Statements}

We suppose \textit{throughout the text} that the following assumptions hold. So in the statements of the results, we will not always precise which among them are required.

\noindent\\
$\left( \mathcal{S}upp \right)$ : The $(X_{j})$, $j \geq 1$ have a common support $\mathcal{S}_{X}=(A,B)$, where $A$ and $B$ may be finite or not.

\noindent\\
$\left( \mathcal{M}gf \right)$ : The mgf's $(\Phi_{j})_{j \geq 1}$ have the same domain of finiteness 
$\Theta=(\alpha, \beta)$, where $\alpha$ and $\beta$ may be finite or not.

\noindent\\
$(\mathcal{H}\kappa)$ : For all $j \geq 1$, $m_{j} := \frac{d\kappa_{j}}{d\theta}$ satisfies $(\mathcal{H})$.

\noindent\\
$(\mathcal{U}f)$ : There exist functions $f_{+}$ and $f_{-}$ which satisfy $(\mathcal{H})$ and such that 
\begin{equation}
\forall j \geq 1, \enskip \forall \theta \in \Theta, \enskip f_{-}(\theta) \leq m_j(\theta) \leq f_{+}(\theta).
\end{equation}

\noindent\\
$\left( \mathcal{C}v \right)$ : For  any compact $K \subset \Theta$,
\begin{equation}
0 < \inf\limits_{j \geq 1} \enskip \inf\limits_{\theta \in K} s_{j}^{2}(\theta) \enskip \leq \enskip 
\sup\limits_{j \geq 1} \enskip \sup\limits_{\theta \in K} \enskip  s_{j}^{2}(\theta) < \infty, 
\end{equation}

\noindent\\
$\left( \mathcal{AM}6 \right)$ : For any compact $K \subset \Theta$,
\begin{equation}
\sup\limits_{j \geq 1} \enskip \sup\limits_{\theta \in K} \enskip |\mu|_{j}^{6} (\theta)  <  \infty.
\end{equation}

\noindent\\
$(\mathcal{C}f)$ : For any $j \geq 1$, $p_{j}$ is a function of class $\mathcal{C}^{1}$ and for any compact $K \subset \Theta$, 
\begin{equation}
\sup\limits_{j \geq 1} \enskip 
\sup\limits_{\theta \in K} \enskip  
\left\| \frac{d \widetilde{p}_{j}^{\theta}}{d x} \right\|_{L^{1}} < \infty.  
\end{equation}

\subsubsection{Elementary Facts}

\begin{fact}
If a function $f$ satisfies $(\mathcal{H})$, then $f$ is a homeomorphism from $(\alpha,\beta)$ to $(A,B)$. 
\end{fact}

\medbreak

\begin{fact} \label{moyenneHomeo}
If a function $f$ is defined as the mean of functions satisfying $(\mathcal{H})$, then $f$ satisfies $(\mathcal{H})$. In particular, $f$ is a homeomorphism from $(\alpha,\beta)$ to $(A,B)$. 
\end{fact}

\medbreak

\begin{corr} \label{homeo}
Let $\ell, n$ be integers with $1 \leq \ell \leq n$. Set
\begin{equation*}
\overline{m}_{\ell,n} := \frac{1}{n-\ell+1} \sum\limits_{j=\ell}^{n} m_{j}. 
\end{equation*}

\noindent
Then, we deduce from $(\mathcal{H}\kappa)$ and Fact $\ref{moyenneHomeo}$ that $\overline{m}_{\ell,n}$ is a homeomorphism from $(\alpha,\beta)$ to $(A,B)$. Consequently, for any $s \in \mathcal{S}_{X}$, the equation 
\begin{equation}
\overline{m}_{\ell,n}(\theta) = s
\end{equation}
 
\noindent
has a unique solution in $\Theta = (\alpha,\beta)$.  
\end{corr}

\medbreak

\begin{defi}
We deduce from Corollary $\ref{homeo}$ that for any $a \in \mathcal{S}_{X}$, for any $n \geq 1$, there exists a unique 
$\theta^{a}_{n} \in \Theta$ s.t. 
\begin{equation*}
\overline{m}_{1,n}(\theta^{a}_{n}) = a. 
\end{equation*}
\end{defi}

\medbreak

\begin{fact}
We deduce from $(\mathcal{H}\kappa)$ that for any $a \in \mathcal{S}_{X}$, there exists a compact set $K_{a}$ of $\mathbb{R}$ s.t. 
\begin{equation}
\left\{  \theta^{a}_{n} : n \geq 1 \right\} \subset K_{a} \subset \Theta. 
\end{equation}
\end{fact}

\medbreak

\begin{corr} \label{mjThetan}
We deduce from the preceding Fact and the Assumptions that, for any $a \in \mathcal{S}_{X}$,  
\begin{equation}
\sup\limits_{n \geq 1} \enskip \sup\limits_{j \geq 1} \enskip |m_{j}(\theta^{a}_{n})| < \infty, 
\end{equation}

\begin{equation}
0 < \inf\limits_{n \geq 1} \enskip \inf\limits_{j \geq 1} \enskip \Phi_{j}(\theta^{a}_{n}) \leq \sup\limits_{n \geq 1} \enskip \sup\limits_{j \geq 1} \enskip \Phi_{j}(\theta^{a}_{n}) < \infty, 
\end{equation}

\begin{equation}
0 < \inf\limits_{n \geq 1} \enskip \inf\limits_{j \geq 1} \enskip s_{j}^{2}(\theta^{a}_{n}) \leq \sup\limits_{n \geq 1} \enskip \sup\limits_{j \geq 1} \enskip s_{j}^{2}(\theta^{a}_{n}) < \infty, 
\end{equation}

\noindent
and for any $3 \leq \ell \leq 6$, 
\begin{equation}
\sup\limits_{n \geq 1} \enskip \sup\limits_{j \geq 1} \enskip |\mu_{j}^{\ell} (\theta^{a}_{n})| \leq \sup\limits_{n \geq 1} \enskip \sup\limits_{j \geq 1} \enskip |\mu|_{j}^{\ell} (\theta^{a}_{n}) < \infty. 
\end{equation}
\end{corr}

\medbreak

\begin{defi} \label{deftin}
We deduce from Corollary $\ref{homeo}$ that for any $n \geq 1$ and $0 \leq i \leq k-1$, there exists a unique $t_{i,n} \in \Theta$ s.t. 
\begin{equation}
\overline{m}_{i+1,n}(t_{i,n}) = \frac{\sum\limits_{j=i+1}^{n} Y_{j}}{n-i}.  
\end{equation}

\noindent
Since $\overline{m}_{i+1,n}$ is a homeomorphism from $\mathcal{S}_{X}$ to $\Theta$, $t_{i,n}$ is a r.v. defined on $(\Omega_{n}, \mathcal{A}_{n})$. 
\end{defi}

\medbreak

\begin{fact} \label{ppteMomentsTin}
Assume that 
\begin{equation}
\label{thetaIEqua}
\max\limits_{0 \leq i \leq k-1} |t_{i,n}| = \mathcal{O}_{\mathcal{P}_{n}}(1)
\end{equation}

\noindent
Then, under the Assumptions, we have that 
\begin{equation}
\max\limits_{0 \leq i \leq k-1} \enskip \sup\limits_{j \geq 1} \enskip |m_{j}(t_{i,n})| =
\mathcal{O}_{\mathcal{P}_{n}}(1), 
\end{equation}

\begin{equation}
\max\limits_{0 \leq i \leq k-1} \enskip \sup\limits_{j \geq 1} \enskip 
\max \left\{ \frac{1}{\Phi_{j}(t_{i,n})} ; \Phi_{j}(t_{i,n}) \right\} =
\mathcal{O}_{\mathcal{P}_{n}}(1),
\end{equation}

\begin{equation}
\label{unExemple}
\max\limits_{0 \leq i \leq k-1} \enskip \sup\limits_{j \geq 1} \enskip 
\max \left\{ \frac{1}{s_{j}^{2}(t_{i,n})} ; s_{j}^{2}(t_{i,n}) \right\} =
\mathcal{O}_{\mathcal{P}_{n}}(1),
\end{equation}

\noindent
and for any $3 \leq \ell \leq 6$, 
\begin{equation}
\max\limits_{0 \leq i \leq k-1} \enskip \sup\limits_{j \geq 1} \enskip |\mu_{j}^{\ell} (t_{i,n})| \leq 
\max\limits_{0 \leq i \leq k-1} \enskip \sup\limits_{j \geq 1} \enskip |\mu|_{j}^{\ell} (t_{i,n})  =  \mathcal{O}_{\mathcal{P}_{n}}(1). 
\end{equation}

\begin{proof}
We prove only $(\ref{unExemple})$, the other proofs being similar. Let $\epsilon > 0$. Then, $(\ref{thetaIEqua})$ implies that there exists $A_{\epsilon} > 0$ s.t. for all $n$ large enough, 
\begin{equation}
\mathcal{P}_{n} \left( \max\limits_{0 \leq i \leq k-1} |t_{i,n}| \leq A_{\epsilon} 
\right) \geq 1 - \epsilon. 
\end{equation}

\noindent
Now, $\left( \mathcal{C}v \right)$ implies that 
\begin{equation}
s_{A_{\epsilon}}^{2} := \sup\limits_{j \geq 1} \enskip \sup\limits_{\theta \in [ - A_{\epsilon} ; A_{\epsilon}]} \enskip s_{j}^{2} (\theta) < \infty. 
\end{equation}

\noindent
Therefore, 
\begin{equation}
\mathcal{P}_{n} \left( \max\limits_{0 \leq i \leq k-1} \enskip \sup\limits_{j \geq 1} \enskip
s_{j}^{2}(t_{i,n}) \leq s_{A_{\epsilon}}^{2} \right) \geq 1 - \epsilon. 
\end{equation}
\end{proof} 

\end{fact}

\begin{rem}
We will prove in Section 3.4. that, under the Assumptions, $(\ref{thetaIEqua})$ holds.  
\end{rem}

\section{Properties of $(Y_{1}^{n})_{n \geq 1}$}

\subsection{Edgeworth expansion} \label{Petrov}

Let  $(X_{j})_{j \geq 1}$ be a sequence of independent r.v.'s with zero means and finite variances. For any 
$j \geq 1$ and $\ell \geq 3$, set 
\begin{equation*}
s_{j}^{2} := \mathbb{E}[X_{j}^{2}] = Var(X_{j}) \quad ; \quad
\sigma_{j} := \sqrt{s_{j}^{2}} \quad ; \quad
\mu_{j}^{\ell} := \mathbb{E}[X_{j}^{\ell}] \quad ; \quad
|\mu|_{j}^{\ell} := \mathbb{E} \left[ \left|X_{j} \right|^{\ell} \right].
\end{equation*}

\noindent
For any $p,q$ with $1 \leq p \leq q$ and $\ell >2$, set 
\begin{equation*}
s_{p,q}^{2} := \sum\limits_{j=p}^{q} s_{j}^{2} \quad ; \quad
\sigma_{p,q} := \sqrt{s_{p,q}^{2}} \quad ; \quad
\mu_{p,q}^{\ell} := \sum\limits_{j=p}^{q} \mu_{j}^{\ell}.
\end{equation*}

\noindent
For any $j \geq 1$, if $p_{j}$ is of class $\mathcal{C}^{1}$, set
\begin{equation*}
d_{j} := \left\| \frac{dp_{j}}{dx} \right\|_{L^{1}}. 
\end{equation*}

\noindent
For $\nu \geq 3$, let $H_{\nu}$ be the Hermite polynomial of degree $\nu$. For example, 
\begin{equation*}
H_{3}(x)=x^{3} - 3x \quad ; \quad H_{4}(x)=x^{4} - 6x^{2} + 3 \quad ; \quad
H_{5}(x)=x^{5} - 10x^{3} + 15x.
\end{equation*}

\begin{theo} \label{classicEdge}
Let $m$ be an integer with $m  \geq 3$. Assume that
\begin{equation}
\label{covClassique}
\sup\limits_{j \geq 1} \enskip \frac{1}{s_{j}^{2}} < \infty, 
\end{equation}

\begin{equation}
\label{AMClassique}
\sup\limits_{j \geq 1} \enskip |\mu|_{j}^{m+1} < \infty, 
\end{equation}

\begin{equation} 
\label{condCF2}
\sup\limits_{j \geq 1} \enskip d_{j} < \infty.
\end{equation}

\noindent\\
Let $\mathfrak{n}$ be the density of the standard normal distribution. For any $n \geq 1$, let $q_{n}$ be the density of $(s_{1,n}^{2})^{-1/2} S_{1,n}$. Then, for all $n$ large enough, we have that  
\begin{equation}
\label{concluPetrov}
\sup\limits_{x \in \mathbb{R}} \enskip 
\left| q_{n}(x) - \mathfrak{n}(x) \left( 1 + \sum\limits_{\nu=3}^{m} P_{\nu,n}(x) \right) \right| = \frac{o(1)}{n^{(m-2)/2}}, 
\end{equation}

\noindent
where, for example, 
\begin{equation*}
P_{3,n}(x) = \frac{\mu_{1,n}^{3}}{6(s_{1,n}^{2})^{3/2}}H_{3}(x)
\end{equation*}

\begin{equation*}
P_{4,n}(x) = \frac{(\mu_{1,n}^{3})^2}{72(s_{1,n}^{2})^{3}} H_{6}(x) + 
\frac{ \mu_{1,n}^{4} - 3\sum\limits_{j=1}^{n} (s_{j}^{2})^2}{24(s_{1,n}^{2})^{2}} H_{4}(x)
\end{equation*}

\begin{equation*}
P_{5,n}(x) = \frac{(\mu_{1,n}^{3})^3}{1296(s_{1,n}^{2})^{9/2}} H_{9}(x) +
\frac{\mu_{1,n}^{3} \left( \mu_{1,n}^{4} - 3\sum\limits_{j=1}^{n} (s_{j}^{2})^2 \right)}
{144(s_{1,n}^{2})^{7/2}} H_{7}(x)
+ \frac{ \mu_{1,n}^{5} - 10\sum\limits_{j=1}^{n} \mu_{j}^{3} s_{j}^{2}}{120(s_{1,n}^{2})^{5/2}} H_{5}(x)
\end{equation*}
\end{theo}

\bigskip

\begin{rem}
\label{remPnu}
We obtain from $(\ref{covClassique})$ and $(\ref{AMClassique})$ that 
\begin{equation}
P_{3,n}(x) = \mathcal{O} \left( \frac{1}{n^{1/2}} \right)H_{3}(x)
\end{equation}

\begin{equation}
P_{4,n}(x) = \mathcal{O} \left( \frac{1}{n} \right)H_{6}(x) + \mathcal{O} \left( \frac{1}{n} \right)H_{4}(x) 
\end{equation}

\begin{equation}
P_{5,n}(x) = \mathcal{O} \left( \frac{1}{n^{3/2}} \right)H_{9}(x) + \mathcal{O} \left( \frac{1}{n^{3/2}} \right)H_{7}(x) + \mathcal{O} \left( \frac{1}{n^{3/2}} \right)H_{5}(x)
\end{equation}
\end{rem}

\subsection{Extensions of the Edgeworth expansion}

For any integers $p,q$ with $1 \leq p \leq q$ and $\theta \in \Theta$, set 
\begin{equation*}
s_{p,q}^{2}(\theta) := \sum\limits_{j=p}^{q} s_{j}^{2}(\theta) \quad ; \quad
\sigma_{p,q}(\theta) := \sqrt{s_{p,q}^{2}(\theta)} \quad ; \quad
\mu_{p,q}^{\ell}(\theta) := \sum\limits_{j=p}^{p} \mu_{j}^{\ell}(\theta).
\end{equation*}

\noindent
For any $j \geq 1$ and $\theta \in \Theta$, set  
\begin{equation*}
d_{j}(\theta) := \left\| \frac{d \widetilde{p}_{j}^{\theta}}{d x} \right\|_{L^{1}}. 
\end{equation*}

\subsubsection{First Extension}

For any $n \geq 1$, let $J_{n}$ be a subset of $\left\{ 1, ..., n \right\}$ s.t. 
$\alpha_{n} := \left| J_{n} \right| < n$. Let $L_{n}$ be the complement of $J_{n}$ in 
$\left\{ 1, ..., n \right\}$. Set 
\begin{equation*}
\overline{S}_{L_{n}} := \sum\limits_{j \in L_{n}} 
\widetilde{X}_{j}^{\theta_{n}^{a}} - \mathbb{E}\left[ \widetilde{X}_{j}^{\theta_{n}^{a}} \right]
= \widetilde{X}_{j}^{\theta_{n}^{a}} - m_{j}(\theta_{n}^{a}). 
\end{equation*}

\noindent
For any $\theta \in \Theta$ and $\ell \geq 3$, set 
\begin{equation*}
s_{L_{n}}^{2}(\theta) := \sum\limits_{j \in L_{n}} s_{j}^{2}(\theta) \quad ; \quad
\sigma_{L_{n}}(\theta) := \sqrt{s_{L_{n}}^{2}(\theta)} \quad ; \quad
\mu_{L_{n}}^{\ell}(\theta) := \sum\limits_{j \in L_{n}} \mu_{j}^{\ell}(\theta). 
\end{equation*}

\begin{theo} 
\label{EdgeTT}
Let $m$ be an integer with $m  \geq 3$. Assume that 
\begin{equation}
\label{covThetan}
\sup\limits_{j \geq 1} \enskip 
\frac{1}{s_{j}^{2}(\theta_{n}^{a})} = \mathcal{O}(1), 
\end{equation}

\begin{equation}
\label{AM6Thetan}
\sup\limits_{j \geq 1} \enskip |\mu|_{j}^{m+1}(\theta_{n}^{a}) = \mathcal{O}(1), 
\end{equation}

\begin{equation}
\label{cf2Thetan}
\sup\limits_{j \geq 1} \enskip d_{j}(\theta_{n}^{a}) = \mathcal{O}(1). 
\end{equation}

\noindent\\
For any $n \geq 1$, let $\overline{q}_{L_{n}}$ be the density of $(s_{L_{n}}^{2})^{-1/2} \overline{S}_{L_{n}}$. Then, for all $n$ large enough, we have that
\begin{equation}
\label{concluTTEdge}
\sup\limits_{x \in \mathbb{R}} \enskip 
\left| \overline{q}_{L_{n}}(x) - \mathfrak{n}(x) \left(1 + \sum\limits_{\nu=3}^{m} \overline{P}_{\nu,L_{n}}(x) \right)\right| = 
\frac{o\left( 1 \right)}{(n-\alpha_{n})^{(m-2)/2}}, 
\end{equation}

\noindent
where the $\overline{P}_{\nu,L_{n}}$ are defined as the $P_{\nu,n}$, except that the $s_{1,n}^{2}$ and the $\mu_{1,n}^{\ell}$ are replaced respectively by $s_{L_{n}}^{2}(\theta_{n}^{a})$ and $\mu_{L_{n}}^{\ell}(\theta_{n}^{a})$.
\end{theo}

\bigskip

\begin{corr}
\label{coroEdgeTT}
Assume that $\left( \mathcal{C}v \right)$, $\left( \mathcal{AM}(m+1) \right)$, $(\mathcal{C}f)$ and $(\mathcal{U}f)$ hold. Then, $(\ref{concluTTEdge})$ holds. 
\end{corr}

\bigbreak

\begin{rem}
By Remark $\ref{remPnu}$, for $\nu = 3, 4, 5$, some $\mathcal{O} \left( \frac{1}{n^{(\nu-2)/2}} \right)$ appear in $P_{\nu,n}$. They are replaced by some $\frac{\mathcal{O}(1)} {(n-\alpha_{n})^{(\nu-2)/2}}$ in 
$\overline{P}_{\nu,L_{n}}$.
\end{rem}

\subsubsection{Second Extension}

\begin{theo}
Let $m$ be an integer with $m  \geq 3$. Assume that

\begin{equation}
\label{covTI}
\max\limits_{0 \leq i \leq k-1} \enskip \sup\limits_{j \geq 1}
 \enskip \frac{1}{s_{j}^{2}(t_{i,n})} = \mathcal{O}_{\mathcal{P}_{n}}(1), 
\end{equation}

\begin{equation}
\label{AM6TI}
\max\limits_{0 \leq i \leq k-1} \enskip \sup\limits_{j \geq 1} \enskip
|\mu|_{j}^{m+1}(t_{i,n}) = \mathcal{O}_{\mathcal{P}_{n}}(1),
\end{equation}

\begin{equation}
\label{cf2TI}
\max\limits_{0 \leq i \leq k-1} \enskip \sup\limits_{j \geq 1} \enskip
d_{j}(t_{i,n}) = \mathcal{O}_{\mathcal{P}_{n}}(1). 
\end{equation}

\noindent\\
Let $e \in \left\{1, 2 \right\}$. We recall that $\overline{q}_{i+e,n}$ is the density of 
$\overline{S}_{i+e,n}$, defined by $(\ref{SienBarre})$. Then, 
\begin{equation}
\label{extensionEdgeTI}
\sup\limits_{x \in \mathbb{R}} \enskip 
\left| \overline{q}_{i+e,n}(x) - 
\mathfrak{n}(x) \left( 1+\sum\limits_{\nu=3}^{m} \overline{P}_{\nu,n}^{(i,e)}(x) \right) \right| 
= \frac{o_{\mathcal{P}_n}(1)}{(n-i-e+1)^{(m-2)/2}},  
\end{equation}

\noindent
where the $\overline{P}_{\nu,n}^{(i,e)}$ are defined as the $P_{\nu,n}$, except that the $s_{1,n}^{2}$ and the 
$\mu_{1,n}^{\ell}$ are replaced respectively by $s_{i+e,n}^{2}(t_{i,n})$ and 
$\mu_{i+e,n}^{\ell}(t_{i,n})$. 
\end{theo}

\begin{proof}
We follow the lines of the proof of Theorem $\ref{classicEdge}$, given in \cite{Petrov 1975}. 
For $j \geq 1$, let $\widetilde{\xi}_{j}$ be the characteristic function of 
$\widetilde{X_{j}}^{t_{i,n}}$. Then, for any $\tau \in \mathbb{R}$, 
\begin{equation}
\widetilde{\xi}_{j}(\tau) =
\int \exp(i\tau x) \frac{\exp(t_{i,n}x)p_{j}(x)}{\Phi_{j}(t_{i,n})}dx 
\end{equation}

\noindent
is a r.v. defined on $(\Omega_{n}, \mathcal{A}_{n})$. Performing a Taylor expansion of $\exp(i\tau x)$, we obtain that 
\begin{equation}
\label{devXij}
\widetilde{\xi}_{j}(\tau) =
1 + \frac{s_{j}^{2}(t_{i,n})}{2} (i\tau)^{2} + 
\sum\limits_{\nu=3}^{m} \frac{\mu_{j}^{\nu}(t_{i,n})}{\nu!} (i\tau)^{\nu} + r_{j}(\tau).  
\end{equation}

\noindent
Then, we deduce from Fact $\ref{ppteMomentsTin}$ that 
\begin{equation}
\label{resteTin}
\sum\limits_{j=i+e}^{n} r_{j} \left( \frac{\tau}{\sigma_{i+e,n}} \right) \leq
\frac{\delta_{i,n}}{(n-i-e+1)^{(m-2)/2}} |\tau|^{m}, \quad \textrm{where} \enskip
\max\limits_{0 \leq i \leq k-1} |\delta_{i,n}| = o_{\mathcal{P}_{n}}(1). 
\end{equation}

\noindent
For any $n \geq 1$, and $\omega \in \Omega_{n}$, we consider a triangular array whose row of index $n$ is composed of the $n-i-e+1$ \textit{independent} r.v.'s
\begin{equation*}
\left( \overline{X}_{j}^{t_{i,n}(\omega)} \right)_{i+e \leq j \leq n}
\end{equation*}

\noindent
Let $\overline{\xi}_{i+e,n}$ be the characteristic function of 
$\overline{S}_{i+e,n}^{t_{i,n}}$, given by $\overline{\xi}_{i+e,n}(\tau) = \int \exp(i\tau x) \overline{q}_{i+e,n}(x) dx$. By independence of the $\left( \overline{X}_{j}^{t_{i,n}(\omega)} \right)_{i+e \leq j \leq n}$ and $(\ref{devXij})$ combined with $(\ref{resteTin})$, we obtain that for suitable some constant $\rho > 0$, for $|\tau| \leq n^{\rho}$, 
\begin{equation}  
\label{controlePetrov}
\left| \overline{\xi}_{i+e,n}(\tau) - u_{m,n}(\tau) \right| \leq 
\frac{\delta_{i,n}}{(n-i-e+1)^{(m-2)/2}} \left(|\tau|^{m} + |\tau|^{3(m-1)} \right) \exp\left(- \frac{\tau^{2}}{2}\right),
\end{equation}

\noindent
where $u_{m,n}$ is the Fourier transform of $\mathfrak{n}(x) \left( 1+\sum\limits_{\nu=3}^{m} \overline{P}_{\nu,n}^{(i,e)}(x) \right)$ and $\max\limits_{0 \leq i \leq k-1} |\delta_{i,n}| = o_{\mathcal{P}_{n}}(1)$. 

\noindent\\
Now, we have that  
\begin{align} 
I &:= \int\limits_{-\infty}^{\infty} 
\left| \overline{\xi}_{i+e,n}(\tau) - u_{m,n}(\tau) \right| d\tau \\
&\leq \label{integralFourier}
 \int\limits_{|\tau| \leq n^{\rho}} 
\left| \overline{\xi}_{i+e,n}(\tau) - u_{m,n}(\tau) \right| d\tau + 
\int\limits_{|\tau| > n^{\rho}} 
\left| u_{m,n}(\tau) \right| d\tau +
\int\limits_{|\tau| > n^{\rho}} 
\left| \overline{\xi}_{i+e,n}(\tau) \right| d\tau. 
\end{align}

\noindent
Then, we obtain from $(\ref{controlePetrov})$ that 
\begin{equation*} 
\int\limits_{|\tau| \leq n^{\rho}} 
\left| \overline{\xi}_{i+e,n}(\tau) - u_{m,n}(\tau) \right| d\tau = 
\frac{o_{\mathcal{P}_{n}}(1)}{(n-i-e+1)^{(m-2)/2}}. 
\end{equation*}

\noindent
Then, using general results on characteristic functions (see Lemma 12 in \cite {Petrov 1975}), we prove that 
\begin{equation}
\int\limits_{|\tau| > n^{\rho}} 
\left| u_{m,n}(\tau) \right| d\tau = 
\frac{o_{\mathcal{P}_{n}}(1)}{(n-i-e+1)^{(m-2)/2}}. 
\end{equation}

\noindent
Now, $(\ref{cf2TI})$ implies that for any $\alpha > 0$ and $\eta > 0$, 
\begin{equation}
\max\limits_{0 \leq i \leq k-1} (n-i-e+1)^{\alpha} \int\limits_{|\tau| > \eta}
\prod\limits_{j=i+e}^{n} \left| \widetilde{\xi}_{j}(\tau) \right| d\tau =
o_{\mathcal{P}_{n}}(1), 
\end{equation}

\noindent
which implies in turn that
\begin{equation}
\int\limits_{|\tau| > n^{\rho}} 
\left| \overline{\xi}_{i+e,n}(\tau) \right| d\tau =
\frac{o_{\mathcal{P}_{n}}(1)}{(n-i-e+1)^{(m-2)/2}}. 
\end{equation}

\noindent
Considering $(\ref{integralFourier})$, we deduce that   
\begin{equation*}
I = \frac{o_{\mathcal{P}_{n}}(1)}{(n-i-e+1)^{(m-2)/2}}. 
\end{equation*}

\noindent
Then, Fourier inversion yields that 
\begin{equation}
\overline{q}_{i+e,n}(x) - 
\mathfrak{n}(x) \left( 1+\sum\limits_{\nu=3}^{m} \overline{P}_{\nu,n}^{(i,e)}(x) \right) = 
\frac{1}{2\pi} \int\limits_{-\infty}^{\infty} \exp(-i\tau x) (\overline{\xi}_{i+e,n}(\tau) - u_{m,n}(\tau)) d\tau. 
\end{equation}

\noindent
Therefore,  
\begin{equation}
\sup\limits_{x \in \mathbb{R}} \enskip 
\left| \overline{q}_{i+e,n}(x) - 
\mathfrak{n}(x) \left( 1+\sum\limits_{\nu=3}^{m} \overline{P}_{\nu,n}^{(i,e)}(x) \right) \right|  \leq \frac{I}{2\pi} = 
\frac{o_{\mathcal{P}_{n}}(1)}{(n-i-e+1)^{(m-2)/2}}.
\end{equation}

\end{proof}

\bigskip

\begin{corr}
\label{condPetrovTI}
Assume that $\left( \mathcal{C}v \right)$, $\left( \mathcal{AM}(m+1) \right)$, $(\mathcal{C}f)$ hold, and that 
\begin{equation}
\max\limits_{0 \leq i \leq k-1} |t_{i,n}| = \mathcal{O}_{\mathcal{P}_{n}}(1)
\end{equation}

\noindent
Then, $(\ref{extensionEdgeTI})$ holds. 
\end{corr}

\begin{rem}
By Remark $\ref{remPnu}$, for $\nu = 3, 4, 5$, some $\mathcal{O} \left( \frac{1}{n^{(\nu-2)/2}} \right)$ appear in $P_{\nu,n}$. They are replaced by some $\frac{\mathcal{O}_{\mathcal{P}_{n}} (1)} {(n-i-1)^{(\nu-2)/2}}$ in 
$P_{\nu,n}^{(i,e)}$.
\end{rem}

\subsection{Moments of $Y_{j}$}

\textit{Throughout this Section 3.3, all the tilted densities considered pertain to 
$\theta = \theta^{a}_{n}$, defined by} 
\begin{equation}
\overline{m}_{1,n}(\theta^{a}_{n}) = a.
\end{equation}

\noindent\\
The moments of the $Y_{j}$'s are obtained by integration of the conditional density. As expected, their first order approximations are the moments of $\widetilde{X_{j}}$.

\begin{lem} \label{momentsYj}
\begin{equation}
\label{espYj} 
\max\limits_{1 \leq j \leq n} \left| \mathbb{E}_{\mathcal{P}_{n}}[Y_j] - m_{j}(\theta^{a}_{n}) \right|  = 
\mathcal{O}\left(\frac{1}{\sqrt{n}}\right).
\end{equation}

\end{lem}

\begin{proof}
For any $n \geq 1$ and $1 \leq j \leq n$, we have that 
\begin{equation}
\label{espYjInt}
\mathbb{E}_{\mathcal{P}_{n}}[Y_{j}] = \int x p(X_{j}=x|S_{1,n}=na) dx
= \int x p(\widetilde{X}_{j}=x | \widetilde{S}_{1,n}=na) dx. 
\end{equation}

\noindent
Let $L_{n}=\left\{1, ..., n\right\}\setminus\lbrace{j}\rbrace$. Normalizing, we obtain that 
\begin{equation}
p(\widetilde{X_{j}}=x | \widetilde{S}_{1,n}=na) = \widetilde{p}_{j}(x)
\left( \frac{\sigma_{1,n}(\theta^{a}_{n})}{\sigma_{L_{n}}(\theta^{a}_{n})} \right)
\frac{p_{\overline{S}_{L_{n}}}(\gamma_{n}^{j}(x))} {p_{\overline{S}_{1,n}}(0)},
\quad \textrm{where} \quad 
\gamma_{n}^{j}(x) := \frac{m_{j}(\theta^{a}_{n})-x}{\sigma_{L_{n}}(\theta^{a}_{n})}.
\end{equation}

\noindent
Since $\left( \mathcal{AM}6 \right)$ implies $\left( \mathcal{AM}4 \right)$, we get from Corollary $\ref{coroEdgeTT}$ with $m=3$ that
\begin{equation}
p_{\overline{S}_{L_{n}}}(\gamma_{n}^{j}(x)) = 
\mathfrak{n}(\gamma_{n}^{j}(x)) \left[ 1 + 
\frac{\mu_{L_{n}}^{3}(\theta^{a}_{n})}{6(s_{L_{n}}^{2}(\theta^{a}_{n}))^{3/2}} H_{3}(\gamma_{n}^{j}(x)) \right]
 + \frac{o(1)}{\sqrt{n-1}}
\end{equation}

\noindent
and 
\begin{equation}
p_{\overline{S}_{1,n}}(0) = \mathfrak{n}(0)+ \frac{o(1)}{\sqrt{n}}.
\end{equation}


\noindent\\
Now, $\left( \mathcal{C}v \right)$, $\left( \mathcal{AM}6 \right)$ and the boundedness of the sequence $(\theta^{a}_{n})_{n \geq 1}$ imply readily that 
\begin{equation}
\frac{\sigma_{1,n}(\theta^{a}_{n})}{\sigma_{L_{n}}(\theta^{a}_{n})} = 
1+\mathcal{O}\left(\frac{1}{n}\right) 
\quad \textrm{and} \quad 
\frac{\mu_{L_{n}}^{3}(\theta^{a}_{n})}{6(s_{L_{n}}^{2}(\theta^{a}_{n}))^{3/2}} = 
\mathcal{O}\left( \frac{1}{\sqrt{n-1}} \right).
\end{equation}

\noindent
Since the functions $\theta \mapsto \mathfrak{n}(\theta)$ and 
$\theta \mapsto \mathfrak{n}(\theta)H_{3}(\theta)$ are bounded, we deduce that
\begin{align}
\frac{p_{\overline{S}_{L_{n}}}(\gamma_{n}^{j}(x))} {p_{\overline{S}_{1,n}}(0)}
&= \left\{\mathfrak{n}(\gamma_{n}^{j}(x)) \left( 1+\mathcal{O}\left( \frac{1}{\sqrt{n}}\right)H_{3}(\gamma_{n}^{j}(x)) \right) +  \frac{o(1)}{\sqrt{n-1}} \right\}
\left\{ \frac{1}{\mathfrak{n}(0)} + \frac{o(1)}{\sqrt{n}} \right\}\\
&= \frac{\mathfrak{n}(\gamma_{n}^{j}(x))}{\mathfrak{n}(0)} + \mathcal{O}\left(\frac{1}{\sqrt{n}}\right)
= \exp \left( -\frac{\gamma_{n}^{j}(x)^{2}}{2} \right) + \mathcal{O}\left(\frac{1}{\sqrt{n}}\right).  
\end{align}

\noindent
Consequently, 
\begin{equation}
\label{densCondThetan}
p(\widetilde{X}_{j}=x|\widetilde{S}_{1,n}=na) = \widetilde{p}_{j}(x)
\left( 1+\mathcal{O}\left(\frac{1}{n}\right) \right)
\left\{ \exp\left(-\frac{\gamma_{n}^{j}(x)^{2}}{2}\right) + \mathcal{O}\left(\frac{1}{\sqrt{n}}\right) \right\}. 
\end{equation}

\noindent
Recalling that $\int x \widetilde{p}_{j}(x)dx = m_{j}(\theta_{n}^{a})$, we deduce from $(\ref{espYjInt})$ and 
$(\ref{densCondThetan})$ that 
\begin{equation}
\mathbb{E}_{\mathcal{P}_{n}}[Y_{j}] = \left\{ \int x \widetilde{p}_{j}(x) \exp\left(-\frac{\gamma_{n}^{j}(x)^{2}}{2}\right) dx +
m_{j}(\theta_{n}^{a}) \mathcal{O}\left(\frac{1}{\sqrt{n}}\right) \right\}
\left( 1+\mathcal{O}\left(\frac{1}{n}\right) \right). 
\end{equation}

\noindent
Therefore, it is enough to prove that
\begin{equation}
\int x \widetilde{p}_{j}(x) \exp\left(-\frac{\gamma_{n}^{j}(x)^{2}}{2}\right) dx = 
m_{j}(\theta_{n}^{a}) + \mathcal{O}\left(\frac{1}{\sqrt{n}}\right)
\end{equation}

\noindent
Now, for any $u \in \mathbb{R}$, 
\begin{equation}
\label{expUdev}
1 - u^{2}/2 \leq  \exp \left( - u^{2}/2  \right)  \leq 1, 
\end{equation}

\noindent
from which we deduce that 
\begin{equation}
\label{xpositif}
\int\limits_{0}^{\infty} x\widetilde{p}_{j}(x)dx -
\frac{1}{2}\int\limits_{0}^{\infty} x\widetilde{p}_{j}(x)\gamma_{n}^{j}(x)^{2}dx \leq
\int\limits_{0}^{\infty} x \widetilde{p}_{j}(x) \exp\left(-\frac{\gamma_{n}^{j}(x)^{2}}{2}\right) dx \leq 
\int\limits_{0}^{\infty} x\widetilde{p}_{j}(x)dx 
\end{equation}

\noindent
and 
\begin{equation}
\label{xnegatif}
\int\limits_{-\infty}^{0} x\widetilde{p}_{j}(x)dx \leq
\int\limits_{-\infty}^{0} x \widetilde{p}_{j}(x) \exp\left(-\frac{\gamma_{n}^{j}(x)^{2}}{2}\right) dx \leq 
\int\limits_{-\infty}^{0} x\widetilde{p}_{j}(x)dx -
\frac{1}{2}\int\limits_{-\infty}^{0} x\widetilde{p}_{j}(x)\gamma_{n}^{j}(x)^{2}dx. 
\end{equation}

\noindent
Adding $(\ref{xpositif})$ and $(\ref{xnegatif})$, we obtain that 
\begin{equation}
\label{encadrementThetan}
m_{j}(\theta_{n}^{a}) - \frac{1}{2}\int\limits_{0}^{\infty} x\widetilde{p}_{j}(x)\gamma_{n}^{j}(x)^{2}dx \leq
\int\limits x \widetilde{p}_{j}(x) \exp\left(-\frac{\gamma_{n}^{j}(x)^{2}}{2}\right) dx \leq 
m_{j}(\theta_{n}^{a}) - \frac{1}{2}\int\limits_{-\infty}^{0} x\widetilde{p}_{j}(x)\gamma_{n}^{j}(x)^{2}dx. 
\end{equation}

\noindent
For any $B \in \mathcal{B}(\mathbb{R})$, we have that
\begin{align}
\int_{B} x \widetilde{p}_{j}(x) \gamma_{n}^{j}(x)^{2} dx &= 
\frac{1}{s_{L_{n}}^{2}(\theta^{a}_{n})}
\left\{ \int_{B} x \widetilde{p}_{j}(x) \left(m_{j}(\theta^{a}_{n})-x \right)^{2}dx \right\}\\
\label{tauxCvThetan}
&= \frac{1}{s_{L_{n}}^{2}(\theta^{a}_{n})}
\sum_{i=0}^{2} \binom{2}{i} m_{j}(\theta^{a}_{n})^{2-i} (-1)^{i} \int_{B} x^{1+i} \widetilde{p}_{j}(x)dx. 
\end{align}

\noindent
Let $i \in \left\{ 0, 1, 2 \right\}$. Recalling that $L_{n}=\left\{1, ..., n\right\}\setminus\lbrace{j}\rbrace$, we get from $\left( \mathcal{C}v \right)$ and $(\mathcal{U}f)$ that 
\begin{equation}
\max\limits_{1 \leq j \leq n} \enskip \frac{1}{s_{L_{n}}^{2}(\theta^{a}_{n})} = 
\mathcal{O} \left( \frac{1}{n} \right)
\qquad \textrm{and} \qquad
\max\limits_{1 \leq j \leq n} \enskip \left| m_{j}(\theta^{a}_{n}) \right|^{2-i} = \mathcal{O}(1). 
\end{equation}

\noindent
Then, $\left( \mathcal{AM}6 \right)$ implies that for all $n \geq 1$, 
\begin{equation}
\max\limits_{1 \leq j \leq n} \enskip \left| \int_{B} x^{1+i} \widetilde{p}_{j}(x)dx \right| \leq
\max\limits_{1 \leq j \leq n} \enskip \int_{\mathbb{R}} |x|^{1+i} \widetilde{p}_{j}(x)dx \leq 
\sup\limits_{j \geq 1} \left\{ 1 + \sup\limits_{\theta \in K_{a}} \mathbb{E} \left[ \left| \widetilde{X}^{\theta}_{j} \right|^{6} \right] \right\} < \infty. 
\end{equation}

\noindent
So we deduce from $(\ref{tauxCvThetan})$ that 
\begin{equation}
\label{intB}
\max\limits_{1 \leq j \leq n} \enskip \int_{B} x \widetilde{p}_{j}(x) \gamma_{n}^{j}(x)^{2} dx = 
\mathcal{O} \left( \frac{1}{n} \right).   
\end{equation}

\noindent
Taking $B = (-\infty, 0)$ and $B = (0, \infty)$ in $(\ref{intB})$, we conclude the proof by 
$(\ref{encadrementThetan})$. 
\end{proof}

\begin{lem} \label{VarCovYj}
We have that 

\begin{equation}
\label{covYj}
\max\limits_{1 \leq j < j' \leq n} \left| \mathbb{E}_{\mathcal{P}_{n}}[Y_{j}Y_{j'}] - 
m_{j}(\theta^{a}_{n})m_{j'}(\theta^{a}_{n}) \right| = 
\mathcal{O}\left(\frac{1}{\sqrt{n}}\right) 
\end{equation}

\noindent
and
\begin{equation}
\label{varYj}
\max\limits_{1 \leq j \leq n} \left| \mathbb{E}_{\mathcal{P}_{n}}[Y_j^2] - \left( s_{j}^{2}(\theta^{a}_{n}) + m_j(\theta^{a}_{n})^2 \right) \right| = \mathcal{O}\left(\frac{1}{\sqrt{n}}\right).
\end{equation}

\end{lem}

\begin{proof}
For any $1 \leq j < j' \leq n$, we have that 
\begin{align*}
\mathbb{E}_{\mathcal{P}_{n}}[Y_{j}Y_{j'}] = 
\int x x' p\left( \left. \widetilde{X_{j}}=x ; \widetilde{X_{j'}}=x' \right| 
\widetilde{S_{1,n}} = na \right) dxdx'.  
\end{align*}

\noindent
Let $L_{n}=\left\{1, ..., n\right\}\setminus\lbrace{j,j'}\rbrace$. Normalizing, we obtain that 
\begin{equation*}
p\left( \left. \widetilde{X_{j}}=x ; \widetilde{X_{j'}}=x' \right| 
\widetilde{S_{1,n}} = na \right) = 
\widetilde{p}_{j}(x) \widetilde{p}_{j'}(x')
\left( \frac{\sigma_{1,n}(\theta^{a}_{n})}{\sigma_{L_{n}}(\theta^{a}_{n})} \right)
\frac{p_{\overline{S}_{L_{n}}}\left(\Gamma_{n}^{j,j'}(x)\right)} {p_{\overline{S}_{1,n}}(0)},
\end{equation*}

\noindent
where
\begin{equation*}
\Gamma_{n}^{j}(x) := \frac{m_{j}(\theta^{a}_{n}) + m_{j'}(\theta^{a}_{n}) - x - x'}
{\sigma_{L_{n}}(\theta^{a}_{n})}.
\end{equation*}

\noindent\\
Since $\left( \mathcal{AM}4 \right)$ holds, we get from Corollary $\ref{coroEdgeTT}$ with $m=3$ that 
\begin{equation}
p\left( \left. \widetilde{X_{j}}=x ; \widetilde{X_{j'}}=x' \right| 
\widetilde{S_{1,n}} = na \right) = 
\widetilde{p}_{j}(x) \widetilde{p}_{j'}(x')
\left( 1+\mathcal{O}\left(\frac{1}{n}\right) \right)
\left\{ \exp\left(-\frac{\Gamma_{n}^{j}(x)^{2}}{2}\right) + 
\mathcal{O}\left(\frac{1}{\sqrt{n}}\right) \right\}. 
\end{equation}

\noindent
As in the preceding proof, we get from $(\ref{expUdev})$ (applied to 
$\exp\left(-\frac{\Gamma_{n}^{j}(x)^{2}}{2}\right)$) that, uniformly in $j$,   
\begin{align*}
\mathbb{E}_{\mathcal{P}_{n}}[Y_{j}Y_{j'}] &=
\int xx' \widetilde{p}_{j}(x) \widetilde{p}_{j'}(x') dxdx' +
\mathcal{O}\left(\frac{1}{\sqrt{n}}\right) \\
&= m_{j}(\theta^{a}_{n})m_{j'}(\theta^{a}_{n}) +
\mathcal{O}\left(\frac{1}{\sqrt{n}}\right).   
\end{align*}

\noindent
The proof of $(\ref{varYj})$ is quite similar. 
\end{proof}

\bigskip

\begin{corr} \label{covEtVarYj}
We have that 
\begin{equation}
\max\limits_{1 \leq j < j' \leq n} Cov_{\mathcal{P}_{n}}(Y_{j}, Y_{j}') = 
\mathcal{O}\left(\frac{1}{\sqrt{n}}\right) 
\end{equation}

\noindent
and
\begin{equation}
\label{covYj}
\max\limits_{1 \leq j \leq n} \left| Var_{\mathcal{P}_{n}}(Y_{j}) - \left( s_{j}^{2}(\theta^{a}_{n}) \right) \right| = \mathcal{O}\left(\frac{1}{\sqrt{n}}\right). 
\end{equation}

\end{corr}

\begin{proof}
We deduce from the preceding Lemmas that for any $1 \leq j < j' \leq n$,
\begin{align*}
Cov_{\mathcal{P}_{n}}(Y_{j}, Y_{j}') &=
\mathbb{E}_{\mathcal{P}_{n}}[Y_{j}Y_{j'}] - \mathbb{E}_{\mathcal{P}_{n}}[Y_{j}]
\mathbb{E}_{\mathcal{P}_{n}}[Y_{j'}] \\
&= \left( m_{j}(\theta^{a}_{n})m_{j'}(\theta^{a}_{n}) + \mathcal{O}\left(\frac{1}{\sqrt{n}}
\right) \right) - \left( m_{j}(\theta^{a}_{n})m_{j'}(\theta^{a}_{n}) + 
\mathcal{O}\left(\frac{1}{\sqrt{n}} \right) \right) \\
&= \mathcal{O}\left(\frac{1}{\sqrt{n}}\right). 
\end{align*} 
\end{proof}

\subsection{Proof of $\max\limits_{0 \leq i \leq k-1} |t_{i,n}| = \mathcal{O}_{\mathcal{P}_{n}}(1)$}

\noindent\\
For any $n \geq 1$ and $i=0, ..., k-1$, set  
\begin{equation}
V_{i+1,n} := \frac{1}{n-i} \sum\limits_{j=i+1}^{n} Z_{j} \quad \textrm{where} \quad
Z_{j} := Y_{j} - \mathbb{E}[Y_{j}]. 
\end{equation}

\begin{lem} \label{V1n}
We have that
\begin{equation}
\mathbb{E}_{\mathcal{P}_{n}} [V_{1,n}^2] = o(1).
\end{equation} 
\end{lem}

\begin{proof}
We have that 
\begin{equation}
\mathbb{E}_{\mathcal{P}_{n}} [V_{1,n}^2] = \frac{1}{n^{2}}
 \left\{ \sum\limits_{j=1}^{n} Var_{\mathcal{P}_{n}}(Y_{j}) + 
 2 \sum\limits_{1 \leq j <  j'  \leq n} Cov_{\mathcal{P}_{n}}(Y_{j}, Y_{j'}) \right\}.
\end{equation}

\noindent
Then, we get from Corollary $\ref{covEtVarYj}$ that 
\begin{equation}
\mathbb{E}_{\mathcal{P}_{n}} [V_{1,n}^2] = \frac{1}{n^{2}}
 \left\{  \sum\limits_{j=1}^{n} \left[ s_{j}^{2}(\theta^{a}_{n}) + \mathcal{O} \left( \frac{1}{\sqrt{n}} \right) \right] 
+ n(n-1) \mathcal{O} \left( \frac{1}{\sqrt{n}} \right) \right\}. 
\end{equation}

\noindent
We conclude the proof by Corollary $\ref{mjThetan}$ which implies that 
\begin{equation}
\frac{1}{n^{2}} \sum\limits_{j=1}^{n} \left[ s_{j}^{2}(\theta^{a}_{n}) + \mathcal{O} \left( \frac{1}{\sqrt{n}} \right) \right] = o(1). 
\end{equation}
\end{proof}

\begin{lem}
\label{maxVi} 
We have that
\begin{equation}
\max\limits_{0 \leq i \leq k-1} |V_{i+1,n}| = o_{\mathcal{P}_n}(1).
\end{equation}
\end{lem}

\begin{proof}
We follow the lines of Kolmogorov's maximal inequality proof. Let $n \geq 1$ and $i \in \left\{ 0, ...,  k - 1 \right\}$. For any $\delta>0$, set
\begin{equation}
A_{i,n} := \left\{|V_{i+1,n}| \geq \delta  \right\} \bigcap \left( \bigcap\limits_{j=0}^{i-1} \left\{ |V_{j+1,n}| < \delta \right\} 
\right),
\end{equation}

\noindent 
and 
\begin{equation}
A_{n} := \left\{ \max\limits_{0 \leq i \leq k-1} |V_{i+1,n}| \geq \delta  \right\} =
\bigcup_{i=0}^{k-1} A_{i,n}. 
\end{equation}

\noindent
Since the $(A_{i,n})_{0 \leq i \leq k-1}$ are non-overlapping, we have that 
\begin{align}
\mathbb{E}_{\mathcal{P}_{n}} [V_{1,n}^{2}] & \geq \sum_{i=0}^{k-1} \int\limits_{A_{i,n}} V_{1,n}^{2} d \mathcal{P}_{n} \\
&= \sum_{i=0}^{k-1} \int\limits_{A_{i,n}} \left\{ (V_{1,n} - V_{i+1,n}) + V_{i+1,n}  \right\}^{2} d\mathcal{P}_{n} \\
& \geq  2  \sum_{i=0}^{k-1} \int\limits_{A_{i,n}} (V_{1,n} - V_{i+1,n})V_{i+1,n} d\mathcal{P}_{n} + \sum_{i=0}^{k-1} \int\limits_{A_{i,n}} V_{i+1,n}^{2} d\mathcal{P}_{n} \\
& \geq 2  \sum_{i=0}^{k-1} \int\limits_{A_{i,n}} (V_{1,n} - V_{i+1,n})V_{i+1,n} d\mathcal{P}_{n} + \delta^{2} \mathcal{P}_{n} (A_{n}). 
\end{align}

\noindent
By Lemma $\ref{V1n}$, it is enough to prove that 
\begin{equation}
\label{analogueKolmo}
\sum_{i=0}^{k-1} \int\limits_{A_{i,n}} (V_{1,n}-V_{i+1,n})V_{i+1,n} d\mathcal{P}_{n} = o(1) . 
\end{equation}

\noindent
In the proof of Kolmogorov, the corresponding term is equal to 0, by independence of the involved random variables. Similarly $(\ref{analogueKolmo})$ will follow from Corollary 
$\ref{VarCovYj}$, which states that the $(Z_{j})$ are asymptotically uncorrelated. Indeed, we have that    
\begin{equation}
\label{eachSum}
\sum_{i=0}^{k-1} \int\limits_{A_{i,n}} (V_{1,n}-V_{i+1,n})V_{i+1,n} d\mathcal{P}_{n} =
\sum_{i=0}^{k-1} \mathbb{E}_{\mathcal{P}_{n}} \left[\mathbf{1}_{A_{i,n}} V_{1,n}V_{i+1,n}
\right] - \sum_{i=0}^{k-1} \mathbb{E}_{\mathcal{P}_{n}}\left[ \mathbf{1}_{A_{i,n}} V_{1,n}^{2} \right]. 
\end{equation}

\noindent 
Then, it is enough to prove that each sum in the right-hand side of $(\ref{eachSum})$ is a 
$o(1)$. We get readily that 
\begin{equation}
\label{V1nVin}
\mathbb{E}_{\mathcal{P}_{n}} \left[\mathbf{1}_{A_{i,n}} V_{1,n}V_{i+1,n} \right] = 
\frac{1}{n(n-i)}\left\{\sum_{j=i+1}^{n}\mathbb{E}_{\mathcal{P}_{n}}\left[\mathbf{1}_{A_{i,n}} Z_{j}^{2} \right] + \sum_{\substack{1 \leq j \leq n \\ i+1 \leq j' \leq n \\ j \neq j'}} 
\mathbb{E}_{\mathcal{P}_{n}}[\mathbf{1}_{A_{i,n}}Z_{j}Z_{j'}]\right\}
\end{equation}

\noindent
and 
\begin{equation}
\label{VinCarre}
\mathbb{E}_{\mathcal{P}_{n}} \left[\mathbf{1}_{A_{i,n}} V_{i+1,n}^{2} \right] = 
\frac{1}{(n-i)^{2}}\left\{\sum_{j=i+1}^{n}\mathbb{E}_{\mathcal{P}_{n}}\left[\mathbf{1}_{A_{i,n}} Z_{j}^{2} \right] + \sum_{\substack{i+1 \leq j,j' \leq n \\ j \neq j'}} 
\mathbb{E}_{\mathcal{P}_{n}}[\mathbf{1}_{A_{i,n}}Z_{j}Z_{j'}] \right\}. 
\end{equation}

\noindent
Now, the Cauchy-Schwarz inequality applied twice, first in $\mathcal{L}^{2}$ and then in 
$\mathbb{R}^{k}$, implies that 
\begin{align}
\sum_{i=0}^{k-1} \frac{1}{n(n-i)} \sum_{j=i+1}^{n}\mathbb{E}_{\mathcal{P}_{n}}\left[\mathbf{1}_{A_{i,n}} Z_{j}^{2} \right] &\leq 
\frac{1}{n}  \sum\limits_{i=0}^{k-1} \mathcal{P}_{n}(A_{i,n})^{1/2} 
\left( \frac{\sum\limits_{j=i+1}^{n} \mathbb{E}_{\mathcal{P}_{n}}\left[ Z_{j}^{4}\right]^{1/2}}{n-i} \right)  \\
\label{CSRk} & \leq 
\frac{1}{n} \left\{ \sum\limits_{i=0}^{k-1} \mathcal{P}_{n}(A_{i,n}) \right\}^{1/2} 
\left\{ \sum_{i=0}^{k-1} \left( \frac{\sum\limits_{j=i+1}^{n} \mathbb{E}_{\mathcal{P}_{n}}\left[ Z_{j}^{4}\right]^{1/2}}{n-i}\right)^{2} \right\}^{1/2}. 
\end{align}

\noindent
Then, $\left[ \sum\limits_{i=0}^{k-1} \mathcal{P}_{n}(A_{i,n}) \right]^{1/2} = 
\mathcal{P}_{n}(A_{n})^{1/2} \leq 1$ and we obtain from Corollary $\ref{VarCovYj}$ and Fact
$\ref{mjThetan}$ that, for all $i \in \left\{0, ..., k-1 \right\}$, 
\begin{equation}
\label{bornitudeCS}
\left( \frac{\sum\limits_{j=i+1}^{n} \mathbb{E}_{\mathcal{P}_{n}}\left[ Z_{j}^{4}\right]^{1/2}}{n-i}\right)^{2} = \left( \frac{\sum\limits_{j=i+1}^{n} \left\{ \mu_{j}^{4}(\theta_{n}^{a}) 
+ \mathcal{O}\left( \frac{1}{n} \right) \right\}^{1/2} }{n-i}\right)^{2} = \mathcal{O}(1). 
\end{equation}

\noindent
Finally, we deduce from $(\ref{CSRk})$ and $(\ref{bornitudeCS})$ that 
\begin{equation}
\sum_{i=0}^{k-1} \frac{1}{n(n-i)} \sum_{j=i+1}^{n}\mathbb{E}_{\mathcal{P}_{n}}\left[\mathbf{1}_{A_{i,n}} Z_{j}^{2} \right] = \frac{1}{n} \left\{ k \mathcal{O}(1) \right\}^{1/2} = o(1). 
\end{equation}

\noindent\\
We obtain similarly that 
\begin{align}
\sum_{i=0}^{k-1} \frac{1}{(n-i)^{2}} \sum_{j=i+1}^{n}\mathbb{E}_{\mathcal{P}_{n}}\left[\mathbf{1}_{A_{i,n}} Z_{j}^{2} \right] &\leq \mathcal{P}_{n}(A_{n})^{1/2}
\left\{ \sum\limits_{i=0}^{k-1} \frac{1}{(n-i)^{2}} \left( \frac{\sum\limits_{j=i+1}^{n} \mathbb{E}_{\mathcal{P}_{n}}\left[ Z_{j}^{4}\right]^{1/2}}{n-i}\right)^{2} \right\}^{1/2} \\
&= \mathcal{O}(1) \left\{ \sum\limits_{i=0}^{k-1} \frac{1}{(n-i)^{2}} \right\}^{1/2} = o(1). 
\end{align}

\noindent
To conclude, we consider the sums involving $\mathbb{E}_{\mathcal{P}_{n}}[\mathbf{1}_{A_{i,n}}Z_{j}Z_{j'}]$, for $j \neq j'$, in $(\ref{V1nVin})$ and $(\ref{VinCarre})$. The Cauchy-Scwarz inequality brings terms of the form $\mathbb{E}_{\mathcal{P}_{n}}[Z_{j}^{2}Z_{j'}^{2}]$. Clearly, $Z_{j}^{2}$ and $Z_{j'}^{2}$ are similarly asymptotically uncorrelated and thereby, we obtain analogously that 
\begin{equation*}
\min \left\{ \sum_{i=0}^{k-1} \frac{1}{n(n-i)}
\sum_{\substack{1 \leq j \leq n \\ i+1 \leq j' \leq n \\ j \neq j'}} 
\mathbb{E}_{\mathcal{P}_{n}}[\mathbf{1}_{A_{i,n}}Z_{j}Z_{j'}] \quad ; \quad 
\sum_{i=0}^{k-1} \frac{1}{(n-i)^{2}}\sum_{\substack{i+1 \leq j,j' \leq n \\ j \neq j'}} 
\mathbb{E}_{\mathcal{P}_{n}}[\mathbf{1}_{A_{i,n}}Z_{j}Z_{j'}] \right\} = o(1), 
\end{equation*}

\noindent
which ends the proof. 
\end{proof}

\bigskip

\begin{theo}
\label{maxTi}
We have that
\begin{equation}
\label{tiEqua}
\max\limits_{0 \leq i \leq k-1} |t_{i,n}| = \mathcal{O}_{\mathcal{P}_{n}}(1). 
\end{equation}
\end{theo}

\begin{proof}
The triangle inequality implies that for any $n \geq 1$,  
\begin{equation}
\label{inegTriang}
\max\limits_{0 \leq i \leq k-1} \left| \overline{m}_{i+1,n}(t_{i,n}) \right| \leq 
\max\limits_{0 \leq i \leq k-1} \left| V_{i+1,n} \right| +
\max\limits_{0 \leq i \leq k-1} \left| \left(\frac{1}{n-i} \sum_{j=i+1}^{n} \mathbb{E}[Y_{j}] \right) - \overline{m}_{i+1,n}(\theta_{n}^{a})\right| + 
\max\limits_{0 \leq i \leq k-1} \left| \overline{m}_{i+1,n}(\theta_{n}^{a}) \right|. 
\end{equation}

\noindent
We get from Lemma $\ref{maxVi}$ and assumption (E) that 
\begin{equation}
\label{ViPetitTau}
\max\limits_{0 \leq i \leq k-1} |V_{i+1,n}| = o_{\mathcal{P}_n}(1). 
\end{equation}

\noindent
Then, Lemma $\ref{momentsYj}$ implies that 
\begin{equation}
\label{espYjpetitTau}
\max\limits_{0 \leq i \leq k-1} \left| \left(\frac{1}{n-i} \sum_{j=i+1}^{n} \mathbb{E}[Y_{j}] \right) - \overline{m}_{i+1,n}(\theta_{n}^{a})\right| \leq
\max\limits_{0 \leq i \leq k-1} \left\{ 
\frac{1}{n-i} \sum_{j=i+1}^{n} \left| \mathbb{E}[Y_{j}] - m_{j}(\theta_{n}^{a}) \right| \right\} = \mathcal{O}\left(\frac{1}{n}\right). 
\end{equation}

\noindent
Now, Fact $\ref{mjThetan}$ implies that 
\begin{equation}
\label{thetanPetitTau}
\max\limits_{0 \leq i \leq k-1} \left| \overline{m}_{i+1,n}(\theta_{n}^{a}) \right| = \mathcal{O}(1). 
\end{equation}

\noindent
Combining $(\ref{inegTriang})$, $(\ref{ViPetitTau})$, $(\ref{espYjpetitTau})$, and 
$(\ref{thetanPetitTau})$, we obtain that 
\begin{equation}
\label{miGrandTau}
\max\limits_{0 \leq i \leq k-1} \left| \overline{m}_{i+1,n}(t_{i,n}) \right| = 
\mathcal{O}_{\mathcal{P}_n}(1). 
\end{equation}

\noindent
Now, $(\mathcal{H}\kappa)$ implies that for all $i = 0, ..., k-1$, $\overline{m}_{i+1,n}$ is a homeomorphism from $\Theta$ to $\mathcal{S}_{X}$. Then, we get from $(\mathcal{U}f)$ that for all $s \in \mathcal{S}_{X}$, 
\begin{equation}
(f_{+})^{-1}(s) \leq (\overline{m}_{i+1,n})^{-1}(s) \leq (f_{-})^{-1}(s).  
\end{equation} 

\noindent
We deduce that $\mathcal{P}_{n}$ - a.s., 
\begin{equation}
(f_{+})^{-1}(\overline{m}_{i+1,n}(t_{i,n})) \leq t_{i,n} \leq (f_{-})^{-1}(\overline{m}_{i+1,n}(t_{i,n})),   
\end{equation}  

\noindent
which combined to $(\ref{miGrandTau})$ concludes the proof. 
\end{proof}

\subsection{The max of the trajectories}

\textit{Throughout this Section 3.5, all the tilted densities considered pertain to 
$\theta = \theta^{a}_{n}$, defined by} 
\begin{equation}
\overline{m}_{1,n}(\theta^{a}_{n}) = a.
\end{equation}

\bigskip

\begin{lem}\label{maxYj}
We have that
\begin{equation}
\max\limits_{1 \leq j \leq n} |Y_{j}| = \mathcal{O}_{\mathcal{P}_{n}} (\log n).
\end{equation}
\end{lem}

\begin{proof}
For any $n \geq 1$, set $M_{n} := \max\limits_{1 \leq j \leq n} |Y_{j}|$. 
For all $s>0$, we have that
\begin{align}
\mathcal{P}_{n} \left(M_{n} \geq s \right) &\leq 
\sum\limits_{j=1}^{n} \mathcal{P}_{n}(Y_j \leq - s)  +  \mathcal{P}_{n}(Y_j \geq s)\\
& = 
\sum\limits_{j=1}^{n}\int_{-\infty}^{-s} P(\widetilde{X}_{j} =x | \widetilde{S}_{1,n} = na)dx +  \int_{s}^{\infty} P(\widetilde{X}_{j}=x | \widetilde{S}_{1,n} = na)dx.
\end{align}

\noindent
Now, we recall from $(\ref{densCondThetan})$ that  
\begin{equation}
p(\widetilde{X}_{j}=x|\widetilde{S}_{1,n}=na) = \widetilde{p}_{j}(x)
\left( 1+\mathcal{O}\left(\frac{1}{n}\right) \right)
\left\{ \exp\left(-\frac{\gamma_{n}^{j}(x)^{2}}{2}\right) + \mathcal{O}\left(\frac{1}{\sqrt{n}}\right) \right\} = \widetilde{p}_{j}(x) \mathcal{O}(1). 
\end{equation}

\noindent
Consequently, there exists an absolute constant $C > 0$ s.t. for all $n \geq 1$, 
\begin{equation}
\mathcal{P}_{n} \left(M_{n} \geq s \right) \leq C \left\{ \sum\limits_{j=1}^{n} 
P\left(\widetilde{X}_{j} \leq -s \right) + P\left(\widetilde{X}_{j} \geq s \right) \right\}.
\end{equation}

\noindent
We get from Markov's inequality that for any $\lambda > 0$, 
\begin{equation}
P\left( \widetilde{X}_{j} \leq -s \right) = 
P\left( \exp(-\lambda \widetilde{X}_{j}) \geq \exp(\lambda s) \right) \leq 
\mathbb{E}\left[ \exp(-\lambda \widetilde{X}_{j}) \right]\exp(-\lambda s) 
\end{equation}

\noindent
and 
\begin{equation}
P\left( \widetilde{X}_{j} \geq s \right) \leq 
\mathbb{E}\left[ \exp(\lambda \widetilde{X}_{j}) \right] \exp(-\lambda s).  
\end{equation}

\noindent
Then, for any $\lambda \neq 0$,
\begin{equation}
\mathbb{E}\left[ \exp( \lambda \widetilde{X}_{j}) \right] =
\int\displaystyle \exp(\lambda x) \left[ \frac{\exp(\theta_{n}^{a} x) p_{j}(x)}
{\Phi_{j}(\theta_{n}^{a})} dx  \right] =
\frac{\Phi_{j}(\theta_{n}^{a} + \lambda)}{\Phi_{j}(\theta_{n}^{a})}. 
\end{equation}  

\noindent
Therefore, 
\begin{equation}
\mathcal{P}_{n} \left(M_{n} \geq s \right) \leq C \left\{ \sum\limits_{j=1}^{n} 
\frac{\Phi_{j}(\theta_{n}^{a} - \lambda)}{\Phi_{j}(\theta_{n}^{a})} + 
\frac{\Phi_{j}(\theta_{n}^{a} + \lambda)}{\Phi_{j}(\theta_{n}^{a})} \right\} \exp(-\lambda s).
\end{equation}

\noindent
Since the sequence $(\theta_{n}^{a})_{n \geq 1}$ is bounded, we can find $\lambda > 0$ s.t. each of the sequences $(\theta_{n}^{a} - \lambda)_{n \geq 1}$ and 
$(\theta_{n}^{a} + \lambda)_{n \geq 1}$ is included in a compact subset of $\Theta$. Therefore, we deduce that there exists an absolute constant $D$ s.t. 
\begin{equation}
\sup\limits_{n \geq 1} \enskip \sup\limits_{j \geq 1} \enskip \max  
\left\{ \frac{\Phi_{j}(\theta^{a}_{n} - \lambda)}{\Phi_{j}(\theta^{a}_{n})} \enskip ; \enskip
\frac{\Phi_{j}(\theta^{a}_{n} + \lambda)}{\Phi_{j}(\theta^{a}_{n})} \right\} \leq D. 
\end{equation}

\noindent
Therefore, 
\begin{equation}
\mathcal{P}_{n} \left(M_{n} \geq s \right) \leq CD n \exp(-\lambda s) = 
CD \exp\left( \log n - \lambda s\right). 
\end{equation}

\noindent
Consequently, for all sequence $(s_{n})_{n \geq 1}$ s.t. 
$\frac{s_{n}}{\log n} \rightarrow \infty$ as $n \rightarrow \infty$, we have that  
\begin{equation}
\label{covMn}
\mathcal{P}_{n} \left(M_{n} \geq s_{n} \right) \rightarrow 0 \enskip \textrm{as} \enskip 
n \rightarrow \infty. 
\end{equation}

\noindent
Set $Z_{n} := \frac{M_{n}}{\log n}$. For any sequence $(a_n)_{n \geq 1}$ s.t. 
$a_{n} \rightarrow \infty$ as $n \rightarrow \infty$, we have that 
\begin{equation}
\mathcal{P}_{n} \left(Z_{n} \geq a_{n} \right) = 
\mathcal{P}_{n} \left(M_{n} \geq s_{n} \right) \enskip \textrm{where} \enskip 
s_{n} := a_{n} \log n, \enskip \textrm{so that} \enskip 
\frac{s_{n}}{\log n} \rightarrow \infty \enskip \textrm{as} \enskip n \rightarrow \infty. 
\end{equation}

\noindent
Finally, we conclude the proof by applying the following Fact, since we get from $(\ref{covMn})$ that
\begin{equation}
\mathcal{P}_{n} \left(Z_{n} \geq a_{n} \right) \rightarrow 0 \enskip \textrm{as} \enskip 
n \rightarrow \infty. 
\end{equation}

\end{proof}

\begin{fact}
For all $n \geq 1$, let $Z_{n} : (\Omega_{n},\mathcal{A}_{n}, \mathcal{P}_{n}) \longrightarrow \mathbb{R}$ be a r.v. Assume that for any sequence $(a_n)_{n \geq 1}$ s.t. 
$a_{n} \rightarrow \infty$ as $n \rightarrow \infty$, we have that  
$\mathcal{P}_{n} (|Z_{n}| \geq a_{n}) \rightarrow 0$ as $n \rightarrow \infty$. Then,  
\begin{equation}
Z_{n} = \mathcal{O}_{\mathcal{P}_{n}}(1). 
\end{equation}
\end{fact}

\begin{proof}
Suppose that the sequence $(Z_{n})$ is not a $\mathcal{O}_{\mathcal{P}_{n}}(1)$. This means that there exists $\epsilon > 0$ s.t. for all 
$k \in \mathbb{N}$, there exists $n(k) \in \mathbb{N}$ s.t. 
\begin{equation}
\label{nonO1}
\mathcal{P}_{n(k)} (|Z_{n(k)}| \geq k) > \epsilon. 
\end{equation}

\noindent\\
If the sequence $(n(k))_{k}$ is bounded, then there exists a fixed $n_{0} \in \mathbb{N}$ 
and a subsequence $(n(k_{j}))_{j \geq 1}$ such that for all $j \geq 1$, $n(k_{j}) = n_{0}$.
We can clearly assume that $k_{j} \rightarrow \infty$ as $j \rightarrow \infty$, which implies that 
\begin{equation}
\lim\limits_{j \rightarrow \infty} \mathcal{P}_{n(k_{j})} (|Z_{n(k_{j})}| \geq k_{j}) =
\lim\limits_{j \rightarrow \infty} \mathcal{P}_{n_{0}} (|Z_{n_{0}}| \geq k_{j}) = 0, 
\end{equation}

\noindent
which contradicts $(\ref{nonO1})$. 

\noindent\\
If the sequence $(n(k))_k$ is not bounded, then there exists a strictly increasing subsequence $(n(k_{j}))_j$ s.t. $n(k_{j}) \rightarrow \infty$ as $j \rightarrow \infty$. Now, we can define a sequence $(a_{n})$ s.t. for all $j \geq 1$, $a_{n(k_{j})} = k_{j}$. We still can assume that $k_{j} \rightarrow \infty$ as $j \rightarrow \infty$. Therefore, we can assume that $a_{n} \rightarrow \infty$ as $n \rightarrow \infty$, which implies that  
\begin{equation}
\lim\limits_{j \rightarrow \infty} \mathcal{P}_{n(k_{j})} (|Z_{n(k_{j})}| \geq k_{j}) =
\lim\limits_{j \rightarrow \infty} \mathcal{P}_{n(k_{j})} (|Z_{n(k_{j})}| \geq a_{n(k_{j})}) = 0, 
\end{equation}

\noindent
which contradicts $(\ref{nonO1})$.

\end{proof}

\subsection{Taylor expansion}

\begin{lem} \label{devLimLandau}
Let $I$ be an interval of $\mathbb{R}$ containing $0$, of non void interior, and $f : I \longrightarrow \mathbb{R}$ a function of class $C^2$. Let $(U_{n})$ be a sequence of random variables $U_{n} : (\Omega_{n},\mathcal{A}_{n}) \longrightarrow (\mathbb{R}, \mathcal{B}(\mathbb{R}))$ s.t.  
\begin{equation}
U_{n} = o_{\mathcal{P}_{n}}(1). 
\end{equation}

\noindent 
Then, there exists $(B_{n})_{n \geq 1} \in \mathcal{A}_{\rightarrow 1}$ s.t. for any $n \geq 1$,  
\begin{equation}
f(U_{n}) = f(0) +  U_{n}f'(0) + U_{n}^{2} \mathcal{O}_{\mathcal{P}_{n}}(1) 
\quad \textrm{on} \enskip B_{n}. 
\end{equation}

\noindent
Furthermore, if $U_{n}=o_{\mathcal{P}_{n}}(u_n)$, with 
$u_{n} \enskip {\underset{n \infty} {\longrightarrow}} \enskip 0$, then 
\begin{equation}
f(U_{n}) = f(0) + o_{\mathcal{P}_{n}}(u_n) \quad \textrm{on} \enskip B_{n}. 
\end{equation}
\end{lem}

\begin{proof}
Let $\epsilon >0$. Let $\delta >0$ s.t. $(-\delta,\delta) \subset I$. Set
\begin{equation*}
B_{n} := \{|U_{n}| < \delta\}. 
\end{equation*}

\noindent
Since $U_{n} = o_{\mathcal{P}_{n}}(1)$, we have that $(B_{n})_{n \geq 1} \in \mathcal{A}_{\rightarrow 1}$. For any $n \geq 1$, $f(U_{n})$ is well defined on $B_{n}$, and the Taylor-Lagrange formula provides a $C_{n}$ with $|C_{n}| \leq |U_{n}|$, s.t. 
\begin{equation}
f(U_{n}) = f(0) + U_{n}f'(0) + \frac{U_{n}^{2}}{2} f '' (C_{n}). 
\end{equation}

\noindent
Now, $C_{n}$ can be obtained from a dichotomy process, initialized with $U_n$. This implies that for all $n$, $C_n$ is a measurable mapping from $(\Omega_{n},\mathcal{A}_{n})$ to $(\mathbb{R}, \mathcal{B}(\mathbb{R}))$, for $C_n$ is the limit of such mappings. Then, as $|C_{n}| \leq |U_{n}|$ and $f''$ is continuous, we have that 
\begin{equation}
C_{n} \enskip {\underset{\mathcal{P}_{n}} {\longrightarrow}} \enskip 0 \Longrightarrow 
f ''(C_{n}) \enskip {\underset{\mathcal{P}_{n}} {\longrightarrow}} \enskip f''(0) \Longrightarrow
f ''(C_{n})=\mathcal{O}_{\mathcal{P}_{n}}(1). 
\end{equation}

\noindent
Furthermore, if $U_{n}=o_{\mathcal{P}_{n}}(u_n)$ with 
$u_{n} \enskip {\underset{n \infty} {\longrightarrow}} \enskip 0$, then 
$\frac{U_{n}^{2}}{2} f '' (C_{n})$ is also a $o_{\mathcal{P}_{n}}(u_n)$.
\end{proof}

\section{Main Results}

\subsection{Theorem with small $k$}
\begin{theo} \label{smallkTheorem}
Suppose that the Assumptions stated in Section 2.6 hold. Assume that 
\begin{equation}
\label{conditionK}
k \longrightarrow \infty \enskip \textrm{as } \enskip n \longrightarrow \infty
\qquad \textrm{and that} \qquad
k=o(n^{\rho}), \quad \textrm{with} \enskip 0 < \rho < 1/2. 
\end{equation}

\noindent
Then, 
\begin{equation}
\label{myDFavecY}
\left\| Q_{nak} - \widetilde{P}_{1}^{k} \right\|_{TV} \enskip {\underset{n \infty} {\longrightarrow}} \enskip 0, 
\end{equation}

\noindent
where $\widetilde{P}_{1}^{k}$ is the joint distribution of independent r.v.'s 
$\left( \widetilde{X_{j}}^{\theta_{n}^{a}} \right)_{1 \leq j \leq k}$.

\end{theo}

\begin{proof}
We have that
\begin{equation}
\pi_{k}(Y_1^k):= p\left( X_{1}^k = Y_{1}^k | S_{1,n} = na \right) = 
\frac{p_{\widetilde{X}_{1}^{k}}\left(Y_{1}^k \right)p_{\widetilde{S}_{k+1,n}}(na-\Sigma_{1,k})}{p_{\widetilde{S}_{1,n}}(na)}. 
\end{equation}

\noindent
Then we normalize, so that 
\begin{equation}
\pi_{k}(Y_1^k) = p_{\widetilde{X}_{1}^{k}}\left(Y_{1}^k \right)
\frac{\sigma_{1,n}(\theta^{a}_{n})}{\sigma_{k+1,n}(\theta^{a}_{n})}
\frac{p_{\overline{S}_{k+1,n}}(Z_{k})}{p_{\overline{S}_{1,n}}(0)} \quad \textrm{where} \quad
Z_{k}:=\frac{\sum\limits_{j=1}^{k}m_{j}(\theta^{a}_{n})-Y_{j}}{\sigma_{k+1,n}(\theta^{a}_{n})}.
\end{equation}

\noindent
Since $\left( \mathcal{AM}4 \right)$ holds, we get from Corollary $\ref{coroEdgeTT}$ with 
$m=3$ that  
\begin{equation}
\pi_{k}(Y_1^k) = p_{\widetilde{X}_{1}^{k}}\left(Y_{1}^k \right)
\frac{\sigma_{1,n}(\theta^{a}_{n})}{\sigma_{k+1,n}(\theta^{a}_{n})}
\frac{\mathfrak{n}(Z_{k}) \left( 1+
\frac{\mu_{k+1,n}^{3}(\theta^{a}_{n})}{6(s_{k+1,n}^{2}(\theta^{a}_{n}))^{3/2}} H_{3}(Z_{k}) \right) + \frac{o(1)}{(n-k)^{3/2}}}
{\mathfrak{n}(0) + \frac{o(1)}{n^{3/2}}}
\end{equation}

\noindent\\
First, we get from Corollary $\ref{mjThetan}$ that 
\begin{equation*}
\frac{\sigma_{1,n}(\theta^{a}_{n})}{\sigma_{k+1,n}(\theta^{a}_{n})} =
\left( 1 + \frac{s_{1,k}^{2}(\theta^{a}_{n})}{s_{k+1,n}^{2}(\theta^{a}_{n})} \right)^{1/2} = 
\left( 1 + \frac{k}{n-k}\mathcal{O}(1) \right)^{1/2} \quad \textrm{and} \quad 
\frac{\mu_{k+1,n}^{3}(\theta^{a}_{n})}{6(s_{k+1,n}^{2}(\theta^{a}_{n}))^{3/2}} = 
\frac{\mathcal{O}(1)}{(n-k)^{1/2}}. 
\end{equation*}

\noindent 
Then, $(\ref{conditionK})$ implies that
\begin{equation}
\frac{\sigma_{1,n}(\theta^{a}_{n})}{\sigma_{k+1,n}(\theta^{a}_{n})} = 1+o(1) 
\quad \textrm{and} \quad  
\frac{\mu_{k+1,n}^{3}(\theta^{a}_{n})}{6(s_{k+1,n}^{2}(\theta^{a}_{n}))^{3/2}} = o(1). 
\end{equation}

\noindent\\
Now, we get from Corollary $\ref{mjThetan}$ and Lemma $\ref{maxYj}$ that 
\begin{equation}
Z_{k} = \frac{k\log n}{\sqrt{n-k}} \mathcal{O}_{\mathcal{P}_{n}}(1).  
\end{equation}

\noindent 
Then, $(\ref{conditionK})$ implies that 
\begin{equation}
Z_{k} = o_{\mathcal{P}_{n}}(1), \quad \textrm{so that} \quad
\mathfrak{n}(Z_{k}) \enskip {\underset{\mathcal{P}_{n}} {\longrightarrow}} \enskip \mathfrak{n}(0) \quad \textrm{and} \quad H_{3}(Z_{k}) \enskip {\underset{\mathcal{P}_{n}} {\longrightarrow}} \enskip H_{3}(0) = 0. 
\end{equation}

\noindent
We obtain from the preceding lines that 
\begin{equation}
\pi_{k}(Y_1^k) = p_{\widetilde{X}_{1}^{k}}\left(Y_{1}^k \right) (1+o_{\mathcal{P}_{n}}(1)). 
\end{equation}

\noindent
Finally, we apply Lemma $\ref{lienISdvt}$ to conclude the proof. 
\end{proof}

\subsection{Theorem with large $k$}

\subsubsection{Statement of the Theorem}

Let $y_{1}^{n} \in \left( \mathcal{S}_{X} \right)^{n}$. Then, for any $0 \leq i \leq k-1$, there exists a unique $\tau_{i}(y_{1}^{n})$ s.t. 
\begin{equation}
\overline{m}_{i+1,n}(\tau_{i}(y_{1}^{n})) = 
\frac{\sum\limits_{j=i+1}^{n} y_{j}}{n-i}. 
\end{equation}

\noindent
For $0 \leq i \leq k-1$, define a density $g(y_{i+1}|y_{1}^{i})$ by 
\begin{equation*}
g(y_{i+1}|y_{1}^{i}) := C_{i}^{-1} \widetilde{p}_{i+1}(y_{i+1}) 
\exp \left(-\frac{\left(y_{i+1} - m_{i+1}(\tau_{i}(y_{1}^{n}))\right)^{2}}
{2s_{i+2,n}^{2}(\tau_{i}(y_{1}^{n}))} \right)
\exp \left( \frac{3\alpha_{i+2,n}^{(3)}(\tau_{i}(y_{1}^{n}))}
{\sigma_{i+2,n}(\tau_{i}(y_{1}^{n}))} y_{i+1} \right),  
\end{equation*}

\noindent
where $C_{i}$ is a normalizing constant which insures that 
$\int g(y_{i+1}|y_{1}^{i}) dy_{i+1} = 1$ and 
\begin{equation*}
\alpha_{i+e,n}^{(3)}(\tau_{i}(y_{1}^{n})) := 
\frac{\mu_{i+e,n}^{3}(\tau_{i}(y_{1}^{n}))}{6(s_{i+e,n}^{2}(\tau_{i}(y_{1}^{n})))^{3/2}}. 
\end{equation*}

\noindent
Then, we define the limiting density on $\mathbb{R}^k$ by 
\begin{equation}
g_{k}(y_{1}^{k}) := \prod_{i=0}^{k-1} g(y_{i+1}|y_{1}^{i}).
\end{equation}

\bigskip

\begin{theo} \label{largekTheorem}
Suppose that the Assumptions stated in Section 2.6 hold. Assume that $k$ is of order 
$n - (\log n)^{\tau}$ with $\tau > 6$.
\begin{equation}
\left\| Q_{nak}-G_{k} \right\|_{TV} \enskip {\underset{n \infty} {\longrightarrow}} \enskip 0,
\end{equation}

\noindent
where $G_{k}$ is the distribution associated to the density $g_{k}$. 

\end{theo}

\begin{proof}
We get from the criterion for convergence in total variation distance stated in Section 2.4. that it is enough to prove the following Theorem. 
\end{proof}

\bigskip

\begin{theo}
Suppose that the Assumptions stated in Section 2.6 hold. Assume that $k$ is of order 
$n - (\log n)^{\tau}$ with $\tau > 6$. Then, there exists $(B_{n})_{n \geq 1} \in \mathcal{A}_{\rightarrow 1}$ s.t. for any $n \geq 1$,  
\begin{equation}
p_{k}(Y_{1}^k) := p(X_{1}^k = Y_{1}^k | S_{1,n} = na) = 
g_{k}(Y_{1}^k) [1+o_{\mathcal{P}_{n}}(1)] \quad \textrm{on} \enskip B_{n}. 
\end{equation}
\end{theo}

\bigskip

\textit{The proof is given hereafter, in three steps. Throughout the proof, all the tilted densities considered pertain to $\theta = t_{i,n}$. We write $s_{j}^{2}$, $\mu_{j}^{\ell}$ instead of $s_{j}^{2}(t_{i,n})$, $\mu_{j}^{\ell}(t_{i,n})$.}

\bigskip

\subsubsection{Identifying $g(Y_{i+1}|Y_{1}^{i})$} 

When $y_{1}^{n} = Y_{1}^{n}$, we have that 
\begin{equation}
\sum\limits_{j=i+1}^{n} y_{j} = 
\sum\limits_{j=i+1}^{n} Y_{j} =
na - \sum\limits_{j=1}^{i} Y_{j} \qquad \mathcal{P}_{n} \enskip \textrm{a.s.}
\end{equation}

\noindent
and
\begin{equation}
\tau_{i}(Y_{1}^{n}) = t_{i,n}. 
\end{equation}

\noindent\\
We recall from the first calculus of Section $\ref{firstCalculus}$ that
\begin{equation}
\pi_{i} = \widetilde{p}_{i+1}(Y_{i+1}) \frac{\sigma_{i+1,n}}{\sigma_{i+2,n}} 
\frac{\overline{q}_{i+2,n}(Z_{i+1})}{\overline{q}_{i+1,n}(0)}, \enskip \textrm{where } 
Z_{i+1} := \frac{m_{i+1} - Y_{i+1}}{\sigma_{i+2,n}}.
\end{equation}

\noindent 
Since $\left( \mathcal{AM}6 \right)$ holds, we get from Corollary $\ref{condPetrovTI}$ with 
$m=5$ that 
\begin{equation}
\label{apresEdge}
\pi_{i} = \widetilde{p}_{i+1}(Y_{i+1}) \frac{\sigma_{i+1,n}}{\sigma_{i+2,n}} 
\left\{ \frac{ \mathfrak{n}(Z_{i+1}) \left[1 + \sum\limits_{\nu = 3}^{5} \overline{P}_{\nu}^{i+2,n}  (Z_{i+1}) \right] + 
\frac{o_{\mathcal{P}_{n}}(1)} {(n-i-1)^{3/2}} }
{\mathfrak{n}(0) \left[1 + \sum\limits_{\nu = 3}^{5} \overline{P}_{\nu}^{i+1,n} (0) \right] + 
\frac{o_{\mathcal{P}_{n}}(1)} {(n-i)^{3/2}} } \right\}.
\end{equation}

\noindent\\
For $e \in \left\{ 1, 2 \right\}$, set 
\begin{equation*}
\alpha_{i+e,n}^{(3)} := \frac{\mu_{i+e,n}^{3}}{6(s_{i+e,n}^{2})^{3/2}} =
\frac{\mathcal{O}_{\mathcal{P}_{n}}(1)} {(n-i-e+1)^{1/2}}, 
\end{equation*}

\begin{equation*}
\beta_{i+e,n}^{(6)} := \frac{(\mu_{i+e,n}^{3})^2}{72(s_{i+e,n}^{2})^{3}}  
= \frac{\mathcal{O}_{\mathcal{P}_{n}}(1)} {n-i-e+1} \quad ; \quad
\beta_{i+e,n}^{(4)} := \frac{ \mu_{i+e,n}^{4} - 3\sum\limits_{j=i+e}^{n} (s_{j}^{2})^2}
{24(s_{i+e,n}^{2})^{2}} = \frac{\mathcal{O}_{\mathcal{P}_{n}}(1)} {n-i-e+1}, 
\end{equation*}

\begin{equation*}
\gamma_{i+e,n}^{(9)} := \frac{(\mu_{i+e,n}^{3})^3}{1296(s_{i+e,n}^{2})^{9/2}} \quad ; \quad
\gamma_{i+e,n}^{(7)} := \frac{\mu_{i+e,n}^{3} \left( \mu_{i+e,n}^{4} - 3\sum\limits_{j=i+e}^{n} (s_{j}^{2})^2 \right)} {144(s_{i+e,n}^{2})^{7/2}} \quad ; \quad
\gamma_{i+e,n}^{(5)} := \frac{ \mu_{i+e,n}^{5} - 10\sum\limits_{j=i+e}^{n} \mu_{j}^{3} s_{j}^{2}}{120(s_{i+e,n}^{2})^{5/2}}, 
\end{equation*}

\noindent
where, for $\ell \in \left\{ 5, 7, 9 \right\}$,
\begin{equation}
\gamma_{i+e,n}^{(\ell)} = \frac{\mathcal{O}_{\mathcal{P}_{n}}(1)} {(n-i-e+1)^{3/2}}. 
\end{equation}

\noindent\\
For $m \in \left\{3, ..., 9 \right\}$, replacing $H_{m}(Z_{i+1})$ by its expression, we have that 
\begin{equation}
\overline{P}_{3}^{i+e,n}(Z_{i+1}) = \alpha_{i+e,n}^{(3)} \left[ Z_{i+1}^{3} - 3Z_{i+1} \right], 
\end{equation}

\begin{equation}
\overline{P}_{4}^{i+e,n}(Z_{i+1}) =
\beta_{i+e,n}^{(6)} \left[ Z_{i+1}^{6} - 15Z_{i+1}^{4} + 45Z_{i+1}^{2} - 15 \right] +
\beta_{i+e,n}^{(4)} \left[Z_{i+1}^{4} - 6Z_{i+1}^{2} + 3 \right],  
\end{equation}

\begin{equation}
\overline{P}_{5}^{i+e,n}(Z_{i+1}) =
\gamma_{i+e,n}^{(9)} \left[ Z_{i+1}^{9}+ ... + 945Z_{i+1} \right] +
\gamma_{i+e,n}^{(7)} \left[ Z_{i+1}^{7} + ... - 105Z_{i+1} \right] +
\gamma_{i+e,n}^{(5)} \left[ Z_{i+1}^{5} + ... + 15Z_{i+1} \right]. 
\end{equation}

\noindent\\
Therefore, 
\begin{equation}
\sum\limits_{\nu = 3}^{5} \overline{P}_{\nu}^{i+2,n}  (Z_{i+1}) = 
- 3 \alpha_{i+2,n}^{(3)} Z_{i+1} - 15 \beta_{i+2,n}^{(6)} + 3 \beta_{i+2,n}^{(4)} + 
\mathcal{O}_{\mathcal{P}_{n}}(1) \frac{(\log n)^{3}}{(n-i-1)^{2}}. 
\end{equation}

\noindent
and 
\begin{equation*}
\sum\limits_{\nu = 3}^{5} \overline{P}_{\nu,n}^{i+1}  (0) = 
 - 15 \beta_{i+1,n}^{(6)} + 3 \beta_{i+1,n}^{(4)}. 
\end{equation*}

\noindent\\
Since $\mathfrak{n}(Z_{i+1}) = \mathcal{O}_{\mathcal{P}_{n}}(1)$, we can factorize $\mathfrak{n}(Z_{i+1})$ in the numerator of the bracket of $(\ref{apresEdge})$, so that 
\begin{equation*}
\pi_{i} = 
\widetilde{p}_{i+1}(Y_{i+1}) \frac{\sigma_{i+1,n}}{\sigma_{i+2,n}}
\frac{ \mathfrak{n}(Z_{i+1}) \left[  1  - 3 \alpha_{i+2,n}^{(3)} Z_{i+1} - 15 \beta_{i+2,n}^{(6)} + 3 \beta_{i+2,n}^{(4)} + 
\mathcal{O}_{\mathcal{P}_{n}}(1) \frac{(\log n)^{3}}{(n-i-1)^{2}} +  
\frac{o_{\mathcal{P}_{n}}(1)} {(n-i-1)^{3/2}} \right] } 
{ \mathfrak{n}(0) \left[ 1 - 15 \beta_{i+1,n}^{(6)} + 3 \beta_{i+1,n}^{(4)} +  \frac{o_{\mathcal{P}_{n}}(1)} {(n-i)^{3/2}}                                     \right] }. 
\end{equation*}

\bigskip

\noindent\\
Since $n-k$ is of order $(\log n)^{\tau}$ with $\tau > 6$, we have for all $n > 1$, and 
$i = 0, ..., k-1$,  
\begin{equation*}
0 \leq \frac{(\log n)^{3}}{(n-i-1)^{2}} (n-i-1)^{3/2} \leq \frac{(\log n)^{3}}{(n-k)^{1/2}}
\longrightarrow 0 \enskip \textrm{ as } n \rightarrow \infty. 
\end{equation*}

\noindent
Therefore, 
\begin{equation*}
\frac{(\log n)^{3}}{(n-i-1)^{2}} = \frac{o(1)}{(n-i-1)^{3/2}}, \enskip \textrm{so that} \enskip \mathcal{O}_{\mathcal{P}_{n}}(1) \frac{(\log n)^{3}}{(n-i-1)^{2}} =  
\frac{o_{\mathcal{P}_{n}}(1)} {(n-i-1)^{3/2}}
\end{equation*}

\noindent\\
Consequently, 
\begin{equation*}
\pi_{i} = \widetilde{p}_{i+1}(Y_{i+1}) \frac{\sigma_{i+1,n}}{\sigma_{i+2,n}} \exp \left( - \frac{Z_{i+1}^{2}}{2}  \right)
\left\{ \frac{ 1 + \frac{3\alpha_{i+2,n}^{(3)}}{\sigma_{i+2,n}} Y_{i+1} - \frac{3\alpha_{i+2,n}^{(3)}}{\sigma_{i+2,n}} m_{i+1} - 15 \beta_{i+2,n}^{(6)} + 3 \beta_{i+2,n}^{(4)} +  \frac{o_{\mathcal{P}_{n}}(1)} {(n-i-1)^{3/2}} } 
{ 1 - 15 \beta_{i+1,n}^{(6)} + 3 \beta_{i+1,n}^{(4)} + \frac{o_{\mathcal{P}_{n}}(1)} {(n-i)^{3/2}} } \right\}
\end{equation*}

\noindent\\
Now, we need to extract $Y_{i+1}$ from the numerator of the bracket hereabove. In that purpose, set
\begin{equation*}
U_{i,n} := \frac{3\alpha_{i+2,n}^{(3)}}{\sigma_{i+2,n}} Y_{i+1} + U'_{i,n} \quad \textrm{where} \quad 
 U'_{i,n} :=  - \frac{3\alpha_{i+2,n}^{(3)}}{\sigma_{i+2,n}} m_{i+1} - 15 \beta_{i+2,n}^{(6)} + 3 \beta_{i+2,n}^{(4)} +  \frac{o_{\mathcal{P}_{n}}(1)} {(n-i-1)^{3/2}}, 
\end{equation*}

\noindent
and 
\begin{equation*}
V_{i,n} := - 15 \beta_{i+1,n}^{(6)} + 3 \beta_{i+1,n}^{(4)} + \frac{o_{\mathcal{P}_{n}}(1)} {(n-i)^{3/2}}.
\end{equation*}

\bigskip

\begin{fact}
For any $n \geq 1$, let $(W_{i,n})_{0 \leq i \leq k-1}$ be r.v.'s defined on $(\Omega_{n},\mathcal{A}_{n})$ s.t. $\max\limits_{0 \leq i \leq k-1}|W_{i,n}| = o_{\mathcal{P}_{n}}(1)$. Then, there exists $(B_{n})_{n \geq 1} \in \mathcal{A}_{\rightarrow 1}$ s.t. for all $n \geq 1$, we have on $B_{n}$ that for all $i=0, ..., k-1$, 
\begin{equation}
1 + W_{i,n} = \exp(W_{i,n} + W_{i,n}^{2} A_{i,n}) \enskip \textrm{where} \enskip
\max\limits_{0 \leq i \leq k-1} |A_{i,n}|=\mathcal{O}_{\mathcal{P}_n}(1). 
\end{equation}

\end{fact}

\begin{proof}
Let $\epsilon > 0$. For any $ n \geq 1$, set 
\begin{equation*}
B_{n} := \left\{ \max\limits_{0 \leq i \leq k-1} |W_{i,n}| < 1/2 \right\}.
\end{equation*}

\noindent
Since $\max\limits_{0 \leq i \leq k-1}|W_{i,n}| = o_{\mathcal{P}_{n}}(1)$, we have that 
$(B_{n})_{n \geq 1} \in \mathcal{A}_{\rightarrow 1}$. Now, set 
\begin{equation*}
f(x):=\log(1+x).
\end{equation*}

\noindent
Then $f$ satisfies the conditions of Lemma $\ref{devLimLandau}$. Therefore, for all $i=0, ..., k-1$, there exists $C_{i,n}$ with $\max\limits_{0 \leq i \leq k-1} |C_{i,n}| \leq \max\limits_{0 \leq i \leq k-1} |W_{i,n}|$ s.t.
\begin{equation}
f(W_{i,n}) = f(0) + W_{i,n}f'(0) + \frac{W_{i,n}^{2}}{2} f''(C_{i,n})
\end{equation}

\noindent
For $n \geq 1$ and $0 \leq i \leq k-1$, set $A_{i,n} := \frac{1}{2} f''(C_{i,n})$. Now,
$f''(x)= - \frac{1}{(1+x)^2}$. Clearly, for all $x \in (0, \frac{1}{2})$, 
$|f''(x)| \leq \frac{1}{(1-x)^2}$. Therefore, for any $n \geq 1$, we have on $B_{n}$ that  
\begin{equation}
\max\limits_{0 \leq i \leq k-1} |A_{i,n}| \leq 
\frac{1}{\left( 1-\max\limits_{0 \leq i \leq k-1} |C_{i,n}| \right)^2} \leq 
\frac{1}{\left( 1-\max\limits_{0 \leq i \leq k-1} |W_{i,n}| \right)^2}, 
\end{equation}

\noindent
which implies that 
$\max\limits_{0 \leq i \leq k-1} |A_{i,n}| = \mathcal{O}_{\mathcal{P}_n}(1)$. 
\end{proof}

\bigskip

\noindent\\
Since  
\begin{equation}
\max\limits_{0 \leq i \leq k-1}|U_{i,n}| = o_{\mathcal{P}_{n}}(1)
\quad \textrm{and} \quad
\max\limits_{0 \leq i \leq k-1}|V_{i,n}| = o_{\mathcal{P}_{n}}(1),
\end{equation}

\noindent
we have that 
\begin{equation}
\frac{1+U_{i,n}}{1+V_{i,n}} = 
\exp \left(U_{i,n} + U_{i,n}^{2} A_{i,n} - V_{i,n} - V_{i,n}^{2} B_{i,n} \right), 
\end{equation}

\noindent
where 
\begin{equation*}
\max\limits_{0 \leq i \leq k-1} |A_{i,n}| = \mathcal{O}_{\mathcal{P}_n}(1)
\quad \textrm{and} \quad 
\max\limits_{0 \leq i \leq k-1} |B_{i,n}| = \mathcal{O}_{\mathcal{P}_n}(1).   
\end{equation*}

\noindent
Consequently, the preceding Fact implies that there exists $(B_{n})_{n \geq 1} \in \mathcal{A}_{\rightarrow 1}$ s.t. for any $n \geq 1$ and $0 \leq i \leq k-1$, 
\begin{equation*}
\pi_{i} = \Gamma_{i} \quad \textrm{on} \enskip B_{n},   
\end{equation*}

\noindent
where
\begin{equation*}
\Gamma_{i} = \widetilde{p}_{i+1}(Y_{i+1}) 
\exp \left(-\frac{(Y_{i+1} - m_{i+1})^{2}}{2s_{i+2,n}^{2}} \right)
\exp \left( \frac{3\alpha_{i+2,n}^{(3)}}{\sigma_{i+2,n}} Y_{i+1} \right) 
\frac{\sigma_{i+1,n}}{\sigma_{i+2,n}} 
\exp \left\{ U'_{i,n} + U_{i,n}^{2} A_{i,n} - V_{i,n} - V_{i,n}^{2} B_{i,n} \right\}.   
\end{equation*}

\noindent\\
In order to identify $g(Y_{i+1}|Y_{1}^{i})$, we have grouped the factors containing $Y_{i+1}$. Thereby, we obtain a function of $Y_{i+1}$, which we normalize to get a density. Thus, set  
\begin{equation*}
g(Y_{i+1}|Y_{1}^{i}) := C_{i}^{-1} \widetilde{p}_{i+1}(Y_{i+1}) 
\exp \left(-\frac{(Y_{i+1} - m_{i+1})^{2}}{2s_{i+2,n}^{2}} \right)
\exp \left( \frac{3\alpha_{i+2,n}^{(3)}}{\sigma_{i+2,n}} Y_{i+1} \right),  
\end{equation*}

\noindent 
where $C_{i}$ satisfies that  
\begin{equation*} 
C_{i} = \int\displaystyle \exp \left(-\frac{(y - m_{i+1})^{2}}{2s_{i+2,n}^{2}} \right)
\exp \left( \frac{3\alpha_{i+2,n}^{(3)}}{\sigma_{i+2,n}} y \right)
\widetilde{p}_{i+1}(y) dy. 
\end{equation*}

\noindent\\
Therefore, 
\begin{equation} 
\label{identifierLi}
\Gamma_{i} = g(Y_{i+1}|Y_{1}^{i}) 
\left\{ C_{i} \frac{\sigma_{i+1,n}}{\sigma_{i+2,n}} 
\exp \left( U'_{i,n} - V_{i,n} + U_{i,n}^{2} A_{i,n} - V_{i,n}^{2} B_{i,n} \right) \right\}. 
\end{equation}

\noindent\\
Our objective is now to prove that 
\begin{equation} 
\label{resteAprouver}
\prod\limits_{i=0}^{k-1} C_{i} \frac{\sigma_{i+1,n}}{\sigma_{i+2,n}} 
\exp \left( U'_{i,n} - V_{i,n} + U_{i,n}^{2} A_{i,n} - V_{i,n}^{2} B_{i,n} \right) = 
1 + o_{\mathcal{P}_n}(1). 
\end{equation}

\noindent\\
In that purpose, we consider firstly the following result.

\bigskip

\begin{lem} \label{uinZin}
For $n \geq 1$, let $(Z_{i,n})_{0 \leq i \leq k-1}$ be r.v.'s defined on
$(\Omega_{n},\mathcal{A}_{n})$ and $(u_{i,n})_{0 \leq i \leq k-1}$ be a sequence of reals. Assume that $\max\limits_{0 \leq i \leq k-1} |Z_{i,n}| = \mathcal{O}_{\mathcal{P}_{n}}(1)$ and $\sum\limits_{i=0}^{k-1} u_{i,n} \longrightarrow 0$ as $n \rightarrow \infty$. Then, 
\begin{equation*}
\prod\limits_{i=0}^{k-1} \exp\left( u_{i,n} Z_{i,n} \right) = 
1+o_{\mathcal{P}_n} \left( 1 \right).
\end{equation*}

\noindent
Consequently, for any $\alpha \geq 0$ and $\beta > 1$,  
\begin{equation*}
\prod\limits_{i=0}^{k-1} \exp \left( \frac{(\log n)^{\alpha}}{(n-i-1)^{\beta}} Z_{i,n} \right) = 1 + o_{\mathcal{P}_n} \left( 1 \right).
\end{equation*}

\end{lem}

\begin{proof}
It is enough to prove that 
\begin{equation}
\sum\limits_{i=0}^{k-1} u_{i,n} Z_{i,n} = o_{\mathcal{P}_n} \left( 1 \right). 
\end{equation}

\noindent
Let $\epsilon > 0$ and $\delta > 0$. There exists $A_{\epsilon} > 0$ and $N_{\epsilon} > 0$
s.t. for all $n \geq N_{\epsilon}$, 
\begin{equation*}
\mathcal{P}_{n} \left( \max\limits_{0 \leq i \leq k-1} |Z_{i,n}| \leq A_{\epsilon} \right) 
\geq 1-\epsilon. 
\end{equation*}

\noindent
Now, there exists $N_{\epsilon, \delta} > 0$ s.t. for all $n \geq N_{\epsilon, \delta}$, 
\begin{equation*}
\sum\limits_{i=0}^{k-1} \left| u_{i,n} \right| < \frac{\delta}{A_{\epsilon}}. 
\end{equation*}

\noindent
Then, for all $n \geq \max \left\{ N_{\epsilon} ; N_{\epsilon, \delta} \right\}$,    
\begin{align*}
\mathcal{P}_{n} \left( \left| \sum\limits_{i=0}^{k-1} u_{i,n} Z_{i,n} \right| < \delta \right)&\geq \mathcal{P}_{n} \left( \left\{ \sum\limits_{i=0}^{k-1} \left| u_{i,n} \right|  
\left|Z_{i,n} \right| < \delta \right\} \bigcap \left\{ \max\limits_{0 \leq i \leq k-1} |Z_{i,n}| \leq A_{\epsilon} \right\} \right) \\
&\geq \mathcal{P}_{n} \left( \left\{ \sum\limits_{i=0}^{k-1} \left| u_{i,n} 
\right| < \frac{\delta}{A_{\epsilon}} \right\} \bigcap \left\{ \max\limits_{0 \leq i \leq k-1} |Z_{i,n}| \leq A_{\epsilon} \right\} \right) \\
&= \mathcal{P}_{n} \left( \max\limits_{0 \leq i \leq k-1} |Z_{i,n}| \leq 
A_{\epsilon} \right) \\
&\geq 1-\epsilon.
\end{align*}

\end{proof}

\subsubsection{The factors estimated by applying Lemma $\ref{uinZin}$}

\begin{corr}
We have that
\begin{equation} 
\label{corrUinZin}
\prod\limits_{i=0}^{k-1} \left\{ \exp \left( U_{i,n}^{2} A_{i,n} - V_{i,n}^{2} B_{i,n} \right) \right\} = 1 + o_{\mathcal{P}_n}(1).
\end{equation}
\end{corr}

\begin{proof}
We may apply Lemma $\ref{uinZin}$, since 
\begin{equation} 
\label{UinBisVin}
\max\limits_{0 \leq i \leq k-1} \left|U_{i,n}\right| = 
\frac{\log n}{n-i-1} \mathcal{O}_{\mathcal{P}_{n}}(1) \quad \textrm{and} \quad
\max\limits_{0 \leq i \leq k-1} \left|V_{i,n}\right| = 
\frac{\mathcal{O}_{\mathcal{P}_{n}}(1)}{n-i-1}. 
\end{equation}
\end{proof}

\noindent\\
Unfortunately, $(\ref{UinBisVin})$ implies that we can not apply Lemma $\ref{uinZin}$ to 
$U'_{i,n}$ and $V_{i,n}$. However, we have that 
\begin{equation} 
\label{UinBisMoinsVin}
U'_{i,n} - V_{i,n} = 
- \frac{3\alpha_{i+2,n}^{(3)}}{\sigma_{i+2,n}} m_{i+1} +
3 \left( \beta_{i+2,n}^{(4)} - \beta_{i+1,n}^{(4)} \right) - 
15 \left( \beta_{i+2,n}^{(6)} - \beta_{i+1,n}^{(6)} \right) +
\frac{o_{\mathcal{P}_n}(1)}{(n-i-1)^{3/2}}.  
\end{equation}

\noindent\\
Now, 
\begin{align} 
\beta_{i+2,n}^{(4)} - \beta_{i+1,n}^{(4)} &=
\frac{\lambda_{i+2,n}(s_{i+1,n}^{2})^{2} - \lambda_{i+1,n}(s_{i+2,n}^{2})^{2}}
{24(s_{i+1,n}^{2})^{2}(s_{i+2,n}^{2})^{2}}, \quad \textrm{where} \enskip
\lambda_{i+e,n} = \sum\limits_{j=i+e}^{n} \lambda_{j} \enskip \textrm{and} \enskip  
\lambda_{j} = \mu_{j}^{4} - 3(s_{j}^{2})^{2} \\
&= \label{annulation} \frac{ \lambda_{i+2,n}\left[ (s_{i+2,n}^{2})^{2} + 2s_{i+2,n}^{2}s_{i+1} + (s_{i+1}^{2})^{2} \right] - (\lambda_{i+2,n}+\lambda_{i+1}) (s_{i+2,n}^{2})^{2}}
{24(s_{i+1,n}^{2})^{2}(s_{i+2,n}^{2})^{2}} \\
&= \label{beta4} \frac{ \mathcal{O}_{\mathcal{P}_{n}}(1)}{(n-i-1)^{2}}, 
\end{align}

\noindent
since in the numerator of $(\ref{annulation})$, the terms of order $(n-i-1)^{3}$, that is the terms $\lambda_{i+2,n}(s_{i+2,n}^{2})^{2}$, vanish.  

\noindent\\
Similarly, we obtain that 
\begin{equation} 
\label{beta6}
\left( \beta_{i+2,n}^{(6)} - \beta_{i+1,n}^{(6)} \right) =
\frac{ \mathcal{O}_{\mathcal{P}_{n}}(1)}{(n-i-1)^{2}}. 
\end{equation}

\noindent\\
Combining $(\ref{UinBisMoinsVin})$, $(\ref{beta4})$, $(\ref{beta6})$,
we obtain that 
\begin{align} 
\prod\limits_{i=0}^{k-1} \exp \left( U'_{i,n} - V_{i,n} \right) &= 
\left\{ \prod\limits_{i=0}^{k-1} \exp \left(-\frac{3\alpha_{i+2,n}^{(3)}}{\sigma_{i+2,n}} 
m_{i+1} \right) \right\}
\left\{ \prod\limits_{i=0}^{k-1} \exp \left( \frac{ \mathcal{O}_{\mathcal{P}_{n}}(1)}{(n-i-1)^{2}} + \frac{o_{\mathcal{P}_n}(1)}{(n-i-1)^{3/2}} \right) \right\} \\
&= \label{isolerOthers} \left\{ \prod\limits_{i=0}^{k-1} \exp \left(-\frac{3\alpha_{i+2,n}^{(3)}}{\sigma_{i+2,n}} 
m_{i+1} \right) \right\}
\left\{ 1 + o_{\mathcal{P}_n}(1) \right\}, 
\end{align}

\noindent
where the last equality follows from Lemma $\ref{uinZin}$. Notice that 
$\frac{3\alpha_{i+2,n}^{(3)}}{\sigma_{i+2,n}} = \frac{\mathcal{O}_{\mathcal{P}_{n}}(1)}{n-i-1}$, so that the corresponding factor is not in the range of Lemma $\ref{uinZin}$. Finally, $(\ref{corrUinZin})$ and $(\ref{isolerOthers})$ imply that 
\begin{equation*}
\prod\limits_{i=0}^{k-1} C_{i} \frac{\sigma_{i+1,n}}{\sigma_{i+2,n}} 
\exp \left( U'_{i,n} - V_{i,n} + U_{i,n}^{2} A_{i,n} - V_{i,n}^{2} B_{i,n} \right) = 
\left\{ \prod\limits_{i=0}^{k-1} C_{i} \frac{\sigma_{i+1,n}}{\sigma_{i+2,n}} 
\exp \left( - \frac{3\alpha_{i+2,n}^{(3)}}{\sigma_{i+2,n}} m_{i+1} \right) \right\}  
\left\{ 1 + o_{\mathcal{P}_n}(1) \right\}.
\end{equation*}

\bigskip

\subsubsection{The other factors}

Therefore, in order to conclude, it is enough to prove that
\begin{equation}
\prod\limits_{i=0}^{k-1} L_{i,n} = 1+o_{\mathcal{P}_n} \left( 1 \right), 
\quad \textrm{where} \enskip 
L_{i,n} := C_{i} \frac{\sigma_{i+1,n}}{\sigma_{i+2,n}} 
\exp \left(-\frac{3\alpha_{i+2,n}^{(3)}}{\sigma_{i+2,n}} m_{i+1} \right). 
\end{equation}

\bigskip

\begin{fact}
We have that
\begin{equation}
\label{rrb}
\frac{\sigma_{i+1,n}}{\sigma_{i+2,n}} = 1 + \frac{s_{i+1}^{2}}{2s_{i+2,n}^{2}} + \frac{\mathcal{O}_{\mathcal{P}_n}(1)}{(n-i-1)^2}, 
\end{equation}

\noindent
and 
\begin{equation}
\label{rrc}
\exp \left(-\frac{3\alpha_{i+2,n}^{(3)}}{\sigma_{i+2,n}} m_{i+1} \right) = 
1 - \frac{3\alpha_{i+2,n}^{(3)}}{\sigma_{i+2,n}} m_{i+1} + 
\frac{\mathcal{O}_{\mathcal{P}_n}(1)}{(n-i-1)^2}
\end{equation}
\end{fact}

\begin{proof}
We have that 
\begin{equation}
\frac{\sigma_{i+1,n}}{\sigma_{i+2,n}} = 
\left(1 + \frac{s_{i+1}^{2}}{s_{i+2,n}^{2}}\right)^{1/2} \quad \textrm{and} \quad
\frac{s_{i+1}^{2}}{s_{i+2,n}^{2}} = \frac{\mathcal{O}_{\mathcal{P}_n}(1)}{(n-i-1)^2}. 
\end{equation}

\noindent
Therefore, $(\ref{rrb})$ follows readily from Lemma $\ref{devLimLandau}$, applied with the function $f : x \mapsto (1+x)^{1/2}$. Similarly, we get $(\ref{rrc})$ by applying 
Lemma $\ref{devLimLandau}$ with the function $f : x \mapsto \exp(x)$. 
\end{proof}

\begin{lem}
We have that 
\begin{equation}
\label{CiRecap}
C_{i} = 1 + \frac{3\alpha_{i+2,n}^{(3)}}{\sigma_{i+2,n}}m_{i+1} - \frac{s_{i+1}^{2}}{2s_{i+2,n}^{2}} + \frac{\mathcal{O}_{\mathcal{P}_n}(1)}{(n-i-1)^2}. 
\end{equation}

\end{lem}

\begin{proof}
Recall that 
\begin{equation*} 
C_{i} = \int\displaystyle \exp \left(v_{i}(y) \right) \widetilde{p}_{i+1}(y) dy 
\quad \textrm{where} \quad 
v_{i}(y):=-\frac{(y-m_{i+1})^2}{2s_{i+2,n}^2}+\frac{3\alpha_{i+2,n}^{(3)}}{\sigma_{i+2,n}}y. 
\end{equation*}

\noindent
A Taylor expansion implies the existence of $w_{i}(y)$ with $|w_{i}(y)| \leq |v_{i}(y)|$ s.t.
\begin{equation}
\label{rrd}
\exp(v_{i}(y)) = 1 + v_{i}(y) + \frac{v_{i}(y)^2}{2}\exp(w_{i}(y)).
\end{equation}

\noindent
Now, 
\begin{align*}
\int (1+v_{i}(y))\widetilde{p}_{i+1}(y) dy &= \int \left[ 1 - \frac{(y-m_{i+1})^2}{2s_{i+2,n}^2} + \frac{3\alpha_{i+2,n}^{(3)}}{\sigma_{i+2,n}} y \right] \widetilde{p}_{i+1}(y) dy \\
&= \int \widetilde{p}_{i+1}(y)dy - 
\frac{1}{2s_{i+2,n}^2} \int (y-m_{i+1})^2 \widetilde{p}_{i+1}(y)dy +
\frac{3\alpha_{i+2,n}^{(3)}}{\sigma_{i+2,n}} \int y \widetilde{p}_{i+1}(y) dy \\
&= 1 - \frac{s_{i+1}^2}{2s_{i+2,n}^2} + \frac{3\alpha_{i+2,n}^{(3)}}{\sigma_{i+2,n}} m_{i+1}. 
\end{align*}

\noindent\\
Consequently, it is enough to prove the following Fact. 
\end{proof}

\bigskip

\begin{fact} \label{Ji}
We have that   
\begin{equation}
\label{rrf}
J_{i} := \int \frac{v_{i}(y)^2}{2} \exp(w_{i}(y))\widetilde{p}_{i+1}(y) dy = 
\frac{\mathcal{O}_{\mathcal{P}_n}(1)}{(n-i-1)^2}. 
\end{equation}
\end{fact}

\begin{proof}
We have that $|w_{i}(y)| \leq |v_{i}(y)|$. Moreover, $w_{i}(y)$ and $v_{i}(y)$ are actually of the same sign, so that $\exp(w_{i}(y)) \leq 1 + \exp(v_{i}(y))$. Therefore,  
\begin{equation}
J_{i} \leq J_{i}^{(1)}+J_{i}^{(2)}   \enskip \textrm{where} \enskip 
J_{i}^{(1)} := \int \frac{v_{i}(y)^2}{2} \widetilde{p}_{i+1}(y) dy
\enskip \textrm{and} \enskip 
J_{i}^{(2)} := \int \frac{v_{i}(y)^2}{2} \exp(v_{i}(y)) \widetilde{p}_{i+1}(y) dy.  
\end{equation}

\noindent
Now, expanding $v_{i}(y)$, we get readily that 
\begin{equation}
J_{i}^{(1)} = \frac{\mathcal{O}_{\mathcal{P}_n}(1)}{(n-i-1)^2}. 
\end{equation}

\noindent\\
Fix $\epsilon > 0$. 

\noindent\\
Then there exist $\alpha_{\epsilon}$, $\beta_{\epsilon}$, $\gamma_{\epsilon}$ positive and a compact $K_{\epsilon}$ s.t., for all $n$ large enough,  
\begin{equation*}
\mathcal{P}_n \left( B_{n}^{\epsilon} := \bigcap\limits_{i = 0}^{k-1} 
\left\{ t_{i} \in K_{\epsilon} \enskip ; \enskip |m_{i+1}| \leq \alpha_{\epsilon} \enskip ; \enskip \frac{1}{2s_{i+2,n}^2} \leq \frac{\beta_{\epsilon}}{n-i-1} \enskip ; \enskip
\left|\frac{3\alpha_{i+2,n}^{(3)}}{\sigma_{i+2,n}} \right| \leq \frac{\gamma_{\epsilon}}{n-i-1} \right\} \right) \geq 1 - \epsilon.
\end{equation*}

\noindent\\
\textit{The following lines hold on $B_{n}^{\epsilon}$}.

\noindent\\
For all real $y$, we have that  
\begin{equation}
|v_{i}(y)| \leq \frac{\beta_{\epsilon}(|y|+\alpha_{\epsilon})^2}{n-i-1} + \frac{\gamma_{\epsilon} |y|}{n-i-1}
\end{equation}

\noindent
For $|y| \geq \alpha_{\epsilon}$, we have that $|y - m_{i+1}| \geq |y - \alpha_{\epsilon}|$, so that 
\begin{equation}
v_{i}(y) \leq - \frac{\beta_{\epsilon} (y-\alpha_{\epsilon})^2}{n-i-1} + \frac{\gamma_{\epsilon} |y|}{n-i-1}. 
\end{equation}

\noindent
Therefore,  
\begin{align}
J_{i}^{(2)} & \leq \frac{1}{2(n-i-1)^2} \int_{|y| \leq \alpha_{\epsilon}} \left[ \beta_{\epsilon} (|y| + \alpha_{\epsilon})^2 + \gamma_{\epsilon} |y| \right]^{2} \exp(v_{i}(y)) \widetilde{p}_{i+1}(y)dy \\
& + \frac{1}{2(n-i-1)^2} \int_{|y| \geq \alpha_{\epsilon}} \left[ \beta_{\epsilon} (|y| + \alpha_{\epsilon})^2 + \gamma_{\epsilon} |y| \right]^{2} \exp\left( - \frac{\beta_{\epsilon} (y - \alpha_{\epsilon})^2}{n-i-1} + \frac{\gamma_{\epsilon} |y|}{n-i-1} \right) \widetilde{p}_{i+1}(y)dy. 
\end{align}

\noindent
Clearly, on $B_{n}^{\epsilon}$, the first integral hereabove is bounded by a constant $I_{\epsilon}$. For the second integral, an integration by parts and Assumption $(\mathcal{C}f)$ imply that, on $B_{n}^{\epsilon}$, it is also bounded by a constant $L_{\epsilon}$. So, 
\begin{equation}
J_{i}^{(2)} = \frac{\mathcal{O}_{\mathcal{P}_n}(1)}{(n-i-1)^2}, 
\end{equation}

\noindent
which concludes the proof.

\end{proof}

\bigskip

\noindent\\
Combining $(\ref{CiRecap})$, $(\ref{rrb})$ and $(\ref{rrc})$, we obtain that 
\begin{align*}
L_{i,n} &:= C_{i} \frac{\sigma_{i+1,n}}{\sigma_{i+2,n}} 
\exp \left(-\kappa_{i,n} m_{i+1} \right) \quad \textrm{where} \quad 
\kappa_{i,n} := \frac{3\alpha_{i+2,n}^{(3)}}{\sigma_{i+2,n}} \\
&= \left[ 1 + \kappa_{i,n} m_{i+1} - \frac{s_{i+1}^{2}}{2s_{i+2,n}^{2}} + \frac{\mathcal{O}_{\mathcal{P}_n}(1)}{(n-i-1)^2} \right] 
\left[ 1 + \frac{s_{i+1}^{2}}{2s_{i+2,n}^{2}} + \frac{\mathcal{O}_{\mathcal{P}_n}(1)}{(n-i-1)^2} \right]
\left[ 1 - \kappa_{i,n} m_{i+1} + 
\frac{\mathcal{O}_{\mathcal{P}_n}(1)}{(n-i-1)^2} \right] \\
&= \left[ 1 + \kappa_{i,n} m_{i+1} - \frac{s_{i+1}^{2}}{2s_{i+2,n}^{2}} + \frac{\mathcal{O}_{\mathcal{P}_n}(1)}{(n-i-1)^2} \right]
\left[ 1 - \kappa_{i,n} m_{i+1} + \frac{s_{i+1}^{2}}{2s_{i+2,n}^{2}} + \frac{\mathcal{O}_{\mathcal{P}_n}(1)}{(n-i-1)^2} \right] \\
&= 1 + \frac{\mathcal{O}_{\mathcal{P}_n}(1)}{(n-i-1)^2}. 
\end{align*}

\noindent\\
Therefore, we may write $L_{i,n} = 1 + \frac{W_{i,n}}{(n-i-1)^2}$, where 
$\max\limits_{0 \leq i \leq k-1} |W_{i,n}| = \mathcal{O}_{\mathcal{P}_n}(1)$. Then, we get from Lemma $\ref{devLimLandau}$ applied with $f : x \mapsto \log(1+x)$ that 
\begin{equation*}
\log(L_{i,n}) = \log \left(1 + \frac{W_{i,n}}{(n-i-1)^2} \right) 
= \frac{W_{i,n}}{(n-i-1)^2} + 
\left(\frac{W_{i,n}}{(n-i-1)^2}\right)^{2} \mathcal{O}_{\mathcal{P}_n}(1)
\end{equation*}

\noindent
Therefore, 
\begin{equation}
\log \left( \prod\limits_{i=0}^{k-1} L_{i,n} \right) = 
\sum\limits_{i=0}^{k-1} \log(L_{i}) = 
\mathcal{O}_{\mathcal{P}_n}(1) \sum\limits_{i=0}^{k-1} \frac{1}{(n-i-1)^2} = 
o_{\mathcal{P}_n}(1). 
\end{equation}

\noindent
Consequently,  
\begin{equation}
\prod\limits_{i=0}^{k-1} L_{i,n} = 
1 + o_{\mathcal{P}_n}(1). 
\end{equation}

\noindent
Finally, we have proved that there exists $(B_{n})_{n \geq 1} \in \mathcal{A}_{\rightarrow 1}$
s.t. for any $n \geq 1$,
\begin{equation*}
p_{k}(Y_{1}^{k}) = \prod\limits_{i=0}^{k-1} \Gamma_{i} \quad \textrm{on} \enskip B_{n}
\end{equation*}

\noindent
and 
\begin{equation*}
\prod\limits_{i=0}^{k-1} \Gamma_{i} = g_{k}(Y_{1}^{k}) \left[1 + o_{\mathcal{P}_n}(1)\right].
\end{equation*}

\end{document}